\documentclass[12pt,reqno]{amsart}
\usepackage{hyperref}
\usepackage{latexsym}
\usepackage{amssymb,amsfonts,amsmath}
\usepackage{graphicx}
\usepackage{indentfirst}
\usepackage{amsmath}
\usepackage{amsfonts}
\usepackage{amssymb}
\usepackage{comment}
\usepackage{mathtools}
\usepackage{amsthm}
\usepackage{enumerate}
\usepackage{extarrows}
\usepackage{relsize} 
\usepackage{tkz-graph}
\usepackage{dsfont}
\usepackage[backend=biber,
style=alphabetic,
sorting=nty, maxbibnames=99]{biblatex}
\usepackage{enumitem}
\usepackage[cmtip,all]{xy}
\usepackage{verbatim}
\usepackage{tikz-cd}
\usepackage{mathtools}
\usepackage{mathrsfs}
\usepackage{geometry}
\newcommand{\unit}{\mathds{1}}   			

\let\oldtocsection=\tocsection
 
\let\oldtocsubsection=\tocsubsection
 
\let\oldtocsubsubsection=\tocsubsubsection
 
\renewcommand{\tocsection}[2]{\hspace{0em}\oldtocsection{#1}{#2}}
\renewcommand{\tocsubsection}[2]{\hspace{1em}\oldtocsubsection{#1}{#2}}
\renewcommand{\tocsubsubsection}[2]{\hspace{2em}\oldtocsubsubsection{#1}{#2}}

\DeclareMathAlphabet{\mathpzc}{OT1}{pzc}{m}{it}

\ifx\pdfoutput\undefined
\else
\fi
\hypersetup{colorlinks=false,bookmarksopen,bookmarksnumbered,citecolor=blue,
pdfstartview=FitH}

\parskip=\medskipamount
\newcommand{\virgolette}{``}

\newcommand{\sfH}{\mathsf{H}}

\newcommand{\Hom}{\mathsf{Hom}}
\newcommand{\Ass}{\mathsf{Ass}}
\newcommand{\Pen}{\mathsf{Pen}}

\newcommand{\de}{\mathrm{d}}
\newcommand{\hot}{\text{ht}}
\renewcommand{\Im}{\ensuremath{\mathrm{Im}}}

\theoremstyle{definition}
\newtheorem{thm}{Theorem}
\newtheorem{defn}{Definition}
\newtheorem{prop}[thm]{Proposition}
\newtheorem{lem}[thm]{Lemma}

\newtheorem{corollary}[thm]{Corollary}
\newtheorem{remark}{Remark}
\newtheorem{example}{Example}
\newtheorem{remarkl}[remark]{Remarkl}

\mathtoolsset{showonlyrefs}

\begin{document}


\title{An Alternative to Homotopy Transfer for $A_\infty$-Algebras}
\date{}

\author[C. A. Cremonini]{C.\ A.\ Cremonini}
\address[C. A. Cremonini]{Arnold-Sommerfeld-Center for Theoretical Physics, Ludwig-Maximilians-Universit\"at M\"unchen, Theresienstr. 37, D-80333 Munich, Germany}
\email{carlo.alberto.cremonini@gmail.com}

\author[V. E. Marotta]{V.\ E.\ Marotta}
\address[V. E. Marotta]{Università di Trieste, Dipartimento di Matematica e Geoscienze,  Via A. Valerio 12/1, 34127 Trieste, Italy}
\email{vincenzoemilio.marotta@units.it}

\dedicatory{To the memory of Roberto \virgolette Rob" Catenacci, a friend.}

\begin{abstract}
In this work, we propose a novel approach to the homotopy transfer procedure starting from a set of homotopy data such that the first differential complex is a differential graded module over the second one. We show that the module structure may be used to induce an $A_\infty$-algebra on the second differential complex, constructed in a similar fashion to the homotopy transfer $A_\infty$-algebra. We prove that, under certain conditions, the $A_\infty$-algebras obtained with this procedure are quasi-isomorphic to the homotopy transfer one. On the other hand, when the side conditions do not hold, we find that there are cases where the existence of an $A_\infty$-quasi-isomorphism with the homotopy transfer $A_\infty$-algebra is obstructed. In other words, we obtain a new $A_\infty$-algebra on the second complex, inequivalent to the homotopy transfer one. Lastly, we prove that these $A_\infty$-algebras are not infinitesimal Hochschild deformations of the homotopy transfer $A_\infty$-algebra.
\end{abstract}

\maketitle
\setcounter{footnote}{0}

{\baselineskip=12pt
\tableofcontents
}

\section{Introduction}

$\infty$-algebras were introduced by Stasheff \cite{Stasheff}, under the name \emph{strongly homotopy algebras}, to describe \virgolette group-like" topological space. Soon after their first appearance, they found a vast realm of applications, both in mathematics and physics. For instance, they represent a cornerstone in the study of formality of K\"{a}hler manifolds \cite{Deligne:1975, Merkulov1999} and deformation quantisation of Poisson manifolds \cite{Kontsevich:defor}, whereas in physics they appear to be the underlying structure describing quantum field theory \cite{Jurco:2018sby, Jurco:2019yfd, Ciric:2020hfj} and string field theory, see \cite{Witten:1986qs, Zwiebach:1997fe, Erler:2013xta, Cremonini:2019xco, Doubek:2020rbg} and references therein.

\subsection{Motivation}~

In this paper, we aim at constructing an $A_\infty$-algebra by considering two quasi-isomorphic differential complexes such that (a subring of) the first one admits a bimodule structure on the second one. In this way, we may induce an $A_\infty$-algebra on the second complex by a technique that is reminiscent of the one used for the Homotopy Transfer Theorem. The main difference resides in the type of algebraic data that is used to induce an $A_\infty$-algebra on the second complex. In order to motivate why this specific setting became relevant to us, let us discuss the setting of Open Superstring Field Theory as in \cite{Erler:2013xta}. In that context, the authors consider three complexes, the BRST complexes of string fields with picture numbers $0,\, -1$ and $-2$, respectively. For the sake of simplicity, consider the picture numbers as labels for the three complexes. These three complexes are quasi-isomorphic, where the quasi-isomorphisms are called \emph{Picture Changing Operators}, and are bimodules over the picture-number 0 complex, which is also the only one naturally equipped with an algebraic structure. The authors used these structures to define on the picture $-1$ complex a 2-product (interpreted as a vertex of interaction between three string fields) which turned out to be non-associative, thus spoiling the gauge-invariance of the action. The gauge-invariance is recovered by adding higher products (i.e., higher vertices of interactions) that encoded the structure of an $A_\infty$-algebra. Notably, the construction of this $A_\infty$-algebra does not rely on the transfer of algebraic data between complexes. We analyse here a similar setting and study under which conditions one can leverage on module structures to induce algebraic structures from one complex to another. In particular, this construction will heavily rely on a novel type of interplay between algebraic and homotopy data.

Supermanifolds and the quasi-isomorphic complexes of differential and integral forms \cite{Manin:1988ds, Noja:2021xos, Cremonini:2023ccv} give another relevant case at the heart of our motivations. Integral forms are defined as a complex of modules over differential forms but are not equipped with a ring structure. Thus the quest for the induction of algebraic structures on this complex. By using the quasi-isomorphisms and the dga-algebra structure on differential forms one can rely on the usual Homotopy Transfer Theorem to transfer an $A_\infty$-algebra on integral forms. Nonetheless, the very rich structure we just described suggests alternative constructions of $A_\infty$-algebras coming from the module structure. We partly explore this topic here and leave the full discussion to the upcoming work \cite{Cremonini-future}, where we shall use this technology to introduce intersection theory on supermanifolds.

We shall briefly discuss here the basic setting of the $A_\infty$-category, for more details see \cite{Keller, Loday2020}.

\subsection{\texorpdfstring{$A_\infty$}{A-infinity}-Algebras}~

Let $\Bbbk$ be a field and let $A$ be a $\mathbb{Z}$-graded vector space over $\Bbbk$,
\begin{equation}
A  \coloneqq \bigoplus_{ i \in \mathbb{Z}} A_i,
\end{equation} 
where the $A_i$ are $\Bbbk$-vector spaces for every $i$. 
\begin{defn}[\textbf{$A_\infty$-Algebra}] An ${A}_\infty$-algebra is a pair $(A, \{ m_n \}_{n \in \mathbb{N}_+} )$, where $A$ is a $\mathbb{Z}$-graded $\Bbbk$-vector space and $ \{ m_n \}_{n \in \mathbb{N}_+}$ is a collection of homogeneous $\Bbbk$-linear maps of degree $2-n$, 
\begin{equation}
m_n \colon A^{\otimes n} \longrightarrow A,
\end{equation}
such that they satisfy the following condition for every $n \in \mathbb{N}_+$:
\begin{equation} \label{condo}
\sum_{n = r + s + t} (-1)^{r+ st} m_{r+t+1}( \unit^{\otimes r} \otimes m_s \otimes \unit^{\otimes t} ) = 0.
\end{equation} 
The maps $m_n$ are called $n$-products. We will denote it by $(A, m_n).$
\end{defn}

This definition encodes the fact that all of the $n$-products satisfy higher associativity relations \emph{up to homotopy}. For more details, see \cite{Stasheff, Keller, vallette:hal-00757010, Loday2020} and references therein.
 
Note that the maps $\{ m_n \}_{n \in \mathbb{N}_+}$ in the definition of the $A_\infty$-algebra belong to the $\Hom$-space of $\Bbbk$-multilinear maps $\Hom (A^{\otimes n}, A)$ for any $n\geq 0$. The space $\Hom (A^{\otimes n}, A)$ is a chain complex $\Hom (A^{\otimes n}, A) \coloneqq \bigoplus_{k\geq 0} \Hom_k (A^{\otimes n}, A), $ where $\Hom_k (A^{\otimes n}, A)$ is the space of multilinear maps $f \colon A^{\otimes n} \rightarrow A$ of degree $k$, with the differential
\begin{align} 
& \partial_A \colon \Hom (A^{\otimes n} , A) \longrightarrow  \Hom (A^{\otimes n} , A) \\
& \partial_A \varphi  \coloneqq [\de, \varphi] = \de \circ \varphi - (-1)^{|\varphi|} \varphi \circ \de_{A^{\otimes n}}. \label{eqn:del}
\end{align}
for any $\varphi \in \Hom (A^{\otimes n}, A),$ where the differential $\de_A$ is extended to a map $\de_{A^{\otimes n}} \colon A^{\otimes n} \rightarrow A^{\otimes n}$ by
$\de_{A^{\otimes n}} \coloneqq \sum_{n -1 = s+t} 
\unit^{\otimes s} \otimes \de \otimes \unit^{\otimes t}$. 

Lastly, let us recall the notion of morphism between $A_\infty$-algebras. An $A_\infty$\emph{-morphism between} the $A_\infty$-algebras $(A, m_n)$ and $(A', m'_n)$ is a family of maps $f_n \colon A^{\otimes n} \to A',$ for all $n \geq 1$ of degree $1-n$ satisfying the relations
\begin{align} \label{eqn:Ainftymorph}
    \sum_{r+s+t = n} (-1)^{r+st} f_{r+t+1} \left( 
 \unit_A^{\otimes r} \otimes m_s \otimes \unit_A^{\otimes t}\right)  = \sum_{l=1}^n (-1)^u m'_l(f_{i_1} \otimes f_{i_2} \otimes ... \otimes f_{i_l}) \, ,
\end{align}
where $i_1 + i_2 + ...+ i_n= n$ and $u= (l-1)(i_1-1) + (l-2)(i_2-1)+ \ldots + (i_{l-1}-1)$.

\subsection{Homotopy Transfer of Algebraic Data}~ \label{subsect:Homotopytransfer}

The homotopy transfer of algebraic data represents the main source of inspiration for this work. We briefly discuss it here and refer to \cite{Kadeishvili, vallette:hal-00757010, Loday2020} for more details on this construction.

Consider the homotopy data 
\begin{equation}
	\begin{tikzcd}[every arrow/.append style={shift left}]
		\hspace{-.2cm} \left( A^\bullet , \de_A \right) \arrow[loop left, distance=2em, start anchor={[yshift=-0.7ex]west}, end anchor={[yshift=0.7ex]west}, "h_A"] \arrow{r}{Y}  & \left( B^\bullet , \de_B \right) \arrow[l, "Z"] 
	\end{tikzcd}
\end{equation}
with $h_A$ degree $-1$ map satisfying
\begin{align}
    \unit_A - Z \circ Y = \de_A \circ h_A + h_A \circ \de_A \, ,
\end{align}
together with the algebraic data given by an associative product $\wedge$ on $A^\bullet$ making $A^\bullet$ a differential graded associative algebra. By using the quasi-isomorphism $Z,Y$, the product on $A^\bullet$ can be transported to $B^\bullet$ by defining
\begin{align} \label{eqn:m2hot}
    m_2^\hot \coloneqq Y \circ \wedge \circ Z^{\otimes 2} \colon B^{\otimes 2} \to B \, .
\end{align}
It can be shown that, in general, $m_2^\hot$ is not associative, in other words, this transport procedure does not preserve the associativity of the product. Nonetheless, it can easily be seen that the associativity of $m_2^\hot$ is violated in a controlled way, i.e.
\begin{align}
    m_2^\hot(m_2^\hot(b \otimes b') \otimes b'') - m_2^\hot (b \otimes m_2^\hot(b' \otimes b'')) =  \partial_B m_3^\hot(b \otimes b' \otimes b'') \, ,
\end{align}
for all $b, b', b'' \in B,$ where 
\begin{align} \label{eqn:m3hot}
    m_3^\hot \coloneqq - Y \circ \wedge \circ \bigl( (h_A \circ \wedge \circ Z^{\otimes 2}) \otimes Z \bigr) + Y \circ \wedge \circ \bigl( Z \otimes (h_A \circ \wedge \circ Z^{\otimes 2}) \bigr)
\end{align}
and $\partial_B$ is defined as in Equation \eqref{eqn:del}.
This construction can be iterated and it is possible to show that $m_2^\hot$ and $m_3^\hot$ are associative up to an exact term with respect to $\partial_B.$ This leads to stating the following remarkable result, proved in \cite{Kadeishvili} for $B = \sfH_\bullet (A)$.

\begin{thm}[\textbf{Homotopy Transfer Theorem}]
    Consider the homotopy data introduced above and let $B^\bullet$ be endowed with the (higher) products $m_n^\hot,$ for all $n \in \mathbb{N}_+,$ obtained as described above. Then $(B^\bullet, m_n^\hot)$ forms an $A_\infty$-algebra. 
\end{thm}

This construction is one of our main guides in order to build a procedure inducing an $A_\infty$-algebra structure on the differential complex $(B^\bullet, \de_B)$ starting from a different type of algebraic data on $A^\bullet$.

\subsection{Outline and Results}~

The outline of the paper is as follows: in Section \ref{sec:HTWMS} we discuss the basic assumptions that will be the main framework of this paper.
In Section \ref{sect:HomotopyTransMod} we show the existence of new  $A_\infty$-algebras coming from an analogue of the homotopy transfer procedure performed under new assumptions. In Section \ref{sec:HochDef} we discuss how this novel procedure relates to homotopy transfer. 

\textbf{Section} \ref{sec:HTWMS} begins with the definition of the main object of this work, the homotopy data given by two quasi-isomorphic cochain complexes $(A, \de_A)$ and $(B, \de_B)$ with quasi-isomorphisms $Y, Z$ such that $A$ is a differential graded associative algebra, $B$ is a differential graded $\Im(Z)$-bimodule and the quasi-isomorphisms $Y,Z$ are $\Im(Z)$-bimodule homomorphisms, where $\Im(Z)$ is assumed to be a subring of $A$. We call this \emph{Homotopy Data with Module Structure}, see Definition \ref{def:homotopydatamodule}. As an example, we show that the de Rham complex on a (split) supermanifold and its quasi-isomorphic complex of integral forms can be endowed with homotopy operators such that they represent an example of Homotopy Data with Module Structure. We thus proceed to show that the above homotopy data satisfying at least one of the so-called side conditions only admit an associative algebra on $B$ coming from the homotopy transfer of the algebra structure on $A,$ i.e. we have the following

\textbf{Proposition} \ref{prop:noAinftyalgebra}. Let $(A, \de_A)$ and $(B, \de_B)$ be homotopy data with module structure as in Definition \ref{def:homotopydatamodule}, endowed with the homotopy transfer $A_\infty$-algebra $(B, m_n^\hot)$. Moreover, suppose that
$h_A$ satisfies the left-side condition. Then the homotopy transfer $A_\infty$-algebra $(B, m_n^\hot)$ reduces to an associative algebra, i.e., $m_n^\hot = 0$ for $n > 2.$ 

This represents one of our motivations for the investigation of an alternative homotopy transfer procedure, so that we may obtain an $A_\infty$-algebra on $B$ admitting higher terms even when the aforementioned side condition holds. 

Inspired by \cite{Markl:2001}, we explore under which circumstances we can preserve the right side condition under a change of homotopy data removing the obstruction to the existence of a strong homotopy equivalence. Moreover, we show that if the quasi-isomorphisms satisfy $Z \circ Y \circ Z= Y,$ there always exists a redefinition of the homotopy operator $h_B$ implementing the right side condition with the same quasi-isomorphisms. We stress that if the right side condition $Z \circ h_B =0$ holds, then we have that the quasi-isomorphisms satisfy $Z \circ Y \circ Z = Z.$ Furthermore, if the right side condition holds, then the obstruction to the existence of a strong homotopy equivalence can always be removed.

We also discuss a weaker notion of the side conditions given by $Y \circ Z \circ h_B = 0$ and $h_B \circ Y \circ Z = 0,$ called \emph{Weak Side Conditions}. We show that if the quasi-isomorphisms satisfy the condition $(Z \circ Y)^2 = Z \circ Y,$ then the homotopy operator can always be redefined in such a way that the weak side conditions hold. The weak side conditions will appear later on in the construction of an $A_\infty$-algebra from the Homotopy Module-Induction.

\textbf{Section} \ref{sect:HomotopyTransMod} is the core of this work and contains the main results concerning the homotopy transfer product induced by a new set of algebraic data.
In particular, in Subsection \ref{subsect:HMT}, we define the main object of our study the Homotopy Module-Induction procedure as in

\textbf{Definition} \ref{def:HTWMSA}.
Let $(A, \de_A)$ and $(B, \de_B)$ be quasi-isomorphic chain complexes endowed with the homotopy data with bimodule structure of Definition \ref{def:homotopydatamodule}. Then we define the  \emph{Homotopy Module-Induced Product (HMIP)}
\begin{align} 
 m_2 ( b \otimes b' ) \coloneqq k_1  Z (b)  \rhd b' + k_2 \, h_B (b) \lhd Z(b') \, ,  
\end{align}
for all $b, b' \in B$ and $k_1,k_2 \in \Bbbk$.

Our analysis starts off by showing that the product $m_2$ differs from the homotopy transfer product $m_2^\hot$  by $\partial_B \tilde{m_2},$ for some $\tilde{m}_2 \in \Hom_{-1}(B^{\otimes 2}, B)$ and is associative up to homotopy, i.e., $\Ass = \partial_B m_3,$ for some $m_3 \in \Hom(B^{\otimes 3}, B),$ see Subsection \ref{subsect:HMTA}. We thus check the first higher associativity relation involving the 3-product $m_3$ and $m_2$ and find out that it is verified in two non-trivial cases: when the coefficients $k_1, \, k_2$ appearing in the definition of $m_2$ are such that $k_1 = - k_2,$ or when the homotopy data satisfies the weak side conditions, see Lemma \ref{lem:Penm4}.   

Before completing the discussion of these two cases, we show that we recover an $A_\infty$-algebra from the Homotopy Module-Induction when the side conditions hold, i.e., we have the following

\textbf{Theorem} \ref{thm:Ainftyside}.
 Let $(A, \de_A), (B, \de_B)$ be homotopy data with module structure as in Definition \ref{def:homotopydatamodule} such that the homotopy $h_B$ satisfies the side conditions of Definition \ref{def:AppA}. 
Then the Homotopy Module-Induction gives an $A_\infty$-algebra $(B, m_n)$ with $m_1= \de_B,$ $m_2$ given by Equation \eqref{HTWMSG}, and higher products
 \begin{align}
      m_n \coloneqq (m_{n-1} \circ m_2) \circ h_B \ , \qquad \text{for } n \geq 3 \, . 
 \end{align}

Thus, we can move on to showing that the Homotopy Module-Induction gives an $A_\infty$-algebra also in the weaker assumptions $k_1=-k_2$ or homotopy data satisfying the weak side conditions. 

\textbf{Theorem} \ref{thm:weaksideorcoefficients}.
   Let $(A, \de_A), (B, \de_B)$ be homotopy data with module structure as in Definition \ref{def:homotopydatamodule} and assume that either of the following conditions hold
   \begin{enumerate}
       \item the homotopy $h_B$ satisfies the weak side conditions; \label{it:mainthm1}
       \item the coefficients appearing in the definition of $m_2$ are such that $k_1 = - k_2.$ \label{it:mainthm2}
   \end{enumerate}
   Then the Homotopy Module-Induction gives the $A_\infty$-algebras $(B, m_n)$ with $m_1 \coloneqq \de_B, $ $m_2$ as in Definition \ref{def:HTWMSA} and higher products
   \begin{align}
       m_n  \coloneqq & (-1)^{n-1} \left( m_2 \circ m_{n-1} \right) \circ_{fb} h_B + \left( m_{n-1} \circ m_2 \right) \circ_{fb} h_B \\
         = & \tilde{m}_2 \circ_Z m_{n-1} + m_{n-1} \circ_{h_B} \tilde{m}_2 \, ,
   \end{align}
where for $n=3$ we consider only the second equality.

\textbf{Section} \ref{sec:HochDef} concentrates on the description of the relation between the homotopy transfer $A_\infty$-algebra and the Homotopy Module-Induced one. In particular, we can find a strict quasi-isomorphims between them when the side conditions hold, as in 

\textbf{Proposition} \ref{prop:HMISCquasiisoA}.
Let Theorem \ref{thm:Ainftyside} hold. Then the $A_\infty$-algebra $(B, m_n)$ is strictly quasi-isomorphic to the dga algebra $(A, \de_A, \wedge)$ seen as an $A_\infty$-algebra if and only if $k_1 \neq - k_2$. The quasi-isomorphism reads $f_1 \coloneqq (k_1 + k_2 ) Z$ and $f_i \coloneqq 0,$ for all $i \neq 1$.

This yields that, in this case, the homotopy transfer $A_\infty$-algebra and the Homotopy Module-Induced one are (strictly) quasi-isomorphic. Moreover, the $A_\infty$-Massey products they induce on the cohomology $\sfH(B)$ are isomorphic. 

When Theorem \ref{thm:weaksideorcoefficients} holds in the case of hypothesis \ref{it:mainthm2}, there is an obstruction to the existence of such quasi-isomorphism, as shown in

\textbf{Proposition} \ref{prop:noquasi-iso}. 
Let Theorem \ref{thm:weaksideorcoefficients} hold in the case of hypothesis \ref{it:mainthm2} and assume that the product $\wedge$ on $A$
induces a non-trivial product on $\sfH (A)$, i.e., $0 \neq [\wedge] \in \sfH \left( \Hom \left( A^{\otimes 2} , A \right) \right)$. Then the cohomology class of $\wedge$ represents an obstruction to the existence of a quasi-isomorphism $f_\bullet:(B,m_n)^{\otimes \bullet} \to (B,m_n^\hot)$.

We also address the problem of determining whether the Homotopy Module-Induced $A_\infty$-algebra is an infinitesimal deformation of the standard homotopy transfer $A_\infty$-algebra, which is stated in the following 

\textbf{Proposition} \ref{prop:notHocdef}.
Assume that Theorem \eqref{thm:Ainftyside} holds and let $(B, m_n^\hot)$ and $(B, m_n)$ be the $A_\infty$-algebras induced by the homotopy transfer and the Homotopy Module-Induction, respectively. Then $(B, m_n)$ is not an infinitesimal deformation of $(B, m_n^\hot).$

A relevant example of this result is given by the case of an associative Homotopy Module-Induced Product.

\subsection{Acknowledgements}~

We thank Branislav Jur\v{c}o, Simone Noja, Ruggero Noris, Ingmar Saberi, Ivo Sachs and Richard Szabo for helpful discussions. The work of {\sc C.A.C.} and {\sc V.E.M.} was supported in part by the GACR Grant EXPRO 19-28268X. The work of C.A.C. is supported by the Excellence Cluster Origins of the DFG under Germany’s Excellence Strategy EXC-2094 390783311. The work of V.E.M. is supported by PNRR MUR projects PE0000023-NQSTI. 

\section{Homotopy Data and Algebraic Structures} \label{sec:HTWMS}

In this Section, we introduce the notion of homotopy data endowed with a bimodule structure and discuss the interplay between these data and the side conditions.
For a concise introduction to homotopy data and homotopy transfer see \cite{Keller}, \cite{Crainic:2004}, \cite{vallette:hal-00757010} and references therein.

\medskip

\subsection{Homotopy Data with Module Structure}~

The cornerstone of this work is the following set of homotopy data. Here and throughout the paper every (graded) vector space is defined over a field $\Bbbk$ of characteristic zero, which will always be omitted from our notation. 

\begin{defn}[\textbf{Homotopy data with bimodule structure}] \label{def:homotopydatamodule}
Let $(A^\bullet, \de_A )$ and $(B^\bullet, \de_B)$ be two quasi-isomorphic cochain complexes together with two degree $-1$ maps (henceforth homotopies) $h_A \colon A^\bullet \to A^{\bullet -1}$, $h_B \colon B^\bullet \to B^{\bullet -1}$, i.e.,
\begin{equation}\label{HTWMSA}
	\begin{tikzcd}[every arrow/.append style={shift left}]
		\hspace{-.2cm} \left( A^\bullet , \de_A \right) \arrow[loop left, distance=2em, start anchor={[yshift=-0.7ex]west}, end anchor={[yshift=0.7ex]west}, "h_A"] \arrow{r}{Y}  & \left( B^\bullet , \de_B \right) \arrow[l, "Z"] \arrow[loop right, distance=2em, start anchor={[yshift=0.7ex]east}, end anchor={[yshift=-0.7ex]east}, "h_B"] 
	\end{tikzcd}
\end{equation}
such that the following conditions hold  
\begin{align}
	\unit_A - Z \circ Y &= \left[ \de_A , h_A \right] \equiv \de_A \circ h_A + h_A \circ \de_A = \partial_A h_A \, , \label{HTWMSB} \\
        \unit_B - Y \circ Z & = \left[ \de_B , h_B \right] \equiv \de_B \circ h_B + h_B \circ \de_B = \partial_B h_B \, . \label{HTWMSBA}
\end{align}
For the sake of notation, in the following, we will denote the complexes by $A^\bullet \equiv A $ and $B^\bullet \equiv B $. Let $A$ and $B$ be endowed also with the following structures:
\begin{enumerate}[leftmargin=*]
\item $(A, \wedge, \de_A)$ is a differential graded associative algebra, i.e., \ $A$ is equipped with a $\Bbbk$-bilinear map
\begin{align}\label{HTWMSC}
	\wedge \colon A \otimes A \rightarrow & \, A \\
 a \otimes a^\prime \mapsto & \, a \wedge a^\prime \ ,
\end{align}
and $\de_A$ satisfies the graded Leibniz rule with respect to $\wedge$:
\begin{align}
    \de_A \left( a \wedge a^\prime \right) = \left( \de_A a \right) \wedge a^\prime + (-1)^{|a|} a \wedge \left( \de_A a^\prime \right) \ ,
\end{align}
where $|a|$ denotes the parity of an homogeneous element $a \in A$;
\vspace{0.5cm}
\item $\mbox{Im} (Z)$ is a subring of $A$ and  $(B, \de_B)$ is a dg-bimodule over $\mbox{Im} (Z)$, i.e., there is a left action of  $\mathrm{Im}(Z)$ on $B$ denoted by
\begin{align}\label{HTWMSD}
	\rhd \, \colon \mathrm{Im}(Z) \otimes B &\to \,  B \\
 Z(b) \otimes b' &\mapsto \,  Z(b) \rhd b' \, , 
\end{align}
and a right action of $\mathrm{Im}(Z)$ on $B$ denoted by
\begin{align}
 \lhd \, \colon B \otimes \mathrm{Im}(Z) &\to B \\
 b \otimes Z(b') & \to b \, \lhd \, Z(b') \, ,
\end{align}
such that the following properties hold:   
\begin{align}\label{HTWMSF}
	Z(b) \rhd \left( Z(b') \rhd b'' \right) &= \left( Z(b) \wedge Z(b') \right) \rhd b'' \ , \\
 \left( b \lhd Z(b') \right) \lhd Z(b'') &= b \lhd \left( Z(b') \wedge Z(b'')  \right) \, ,
\end{align}
for all $b, \, b', \, b'' \in B$, and \begin{align}\label{HTWMSE}
	\de_B \left( Z(b) \rhd b' \right) &= Z\left( \de_B b \right) \rhd b' + \left( -1 \right)^{|b|} Z(b) \rhd \left( \de_B b' \right) \, , \\
 \de_B \left( b \lhd Z(b') \right) &= \de_B b \lhd Z(b') + (-1)^{|b|} \lhd Z(\de_B b') \, ,
\end{align}
for all $b, \, b' \in B$ homogeneous elements;
\vspace{0.5cm}
\item \label{it:ImZhomo} The maps $Y \colon A \rightarrow B$ and $Z \colon B \rightarrow A $ are $\mbox{Im}(Z)$-bimodule homomorphisms, i.e.,
\begin{align}
Y (Z (b) \wedge a) = Z (b) \rhd Y (a) \, , \qquad & Z (Z (b) \rhd b^\prime) = Z(b)\wedge Z (b^\prime) \, , \\
Y ( a \wedge Z(b) ) = Y (a) \lhd Z (b) \, , \qquad & Z ( b \lhd Z (b') ) = Z (b) \wedge Z (b') \, ,
\end{align}
for all $a \in A, \  b, b^\prime \in B$.
\end{enumerate}
The cochain complexes $A,B$ endowed with the above structures are called \emph{homotopy data with bimodule structure}.
\end{defn}

\begin{remark}
It follows from Property (\ref{it:ImZhomo}) that 
\begin{equation}
    Y (Z(b) \wedge Z(b^\prime)) = Z(b) \rhd (Y \circ Z)(b^\prime) = (Y \circ Z(b)) \lhd Z (b^\prime),
\end{equation}
for every $b, b^\prime \in B$, and that
\begin{align}
    Z(b) \rhd b ' - b \lhd Z(b') \in \ker(Z) \, ,
\end{align}
for all $b, \, b' \in B.$ 

If $\wedge$ is graded-commutative, then the left $\Im(Z)$-action and the right $\Im(Z)$-action must agree on $\Im(Y)$ up to the sign coming from the graded-commutativity of $\wedge.$
\end{remark}

We present here one of the motivating examples of the construction described in this paper. We will limit ourselves to describing the homotopy data with module structure as in Definition \eqref{def:homotopydatamodule}; a full analysis with the construction of the $A_\infty$ algebra with the methods described in this paper is deferred to a future work \cite{Cremonini-future} where it will be linked to intersection theory on supermanifolds. The core of this construction relies on the fact that supermanifolds are naturally equipped with two complexes of forms: differential forms (or \emph{superforms}), which are a generalisation of usual differential forms on ordinary manifolds, and \emph{integral forms}, which define sections which are integrable on odd-codimension-0 submanifolds. The former comes naturally with the structure of an algebra, and the latter comes as a module over superforms, but no algebraic structure is given. We will not linger on definitions regarding supermanifolds, but rather just mention the basics which are needed for the present purposes. We refer the reader to \cite{Manin:1988ds, Noja:2021xos} for a detailed introduction to supermanifolds.

\begin{example}[\textbf{Integral Forms on Split Supermanifolds}] \label{eg:supermani}
Recall that a supermanifold $\mathcal{SM}$ of dimension $(m |n)$ is a locally decomposable, locally ringed space $\left( \left| \mathcal{SM} \right| , \mathcal{O}_{\mathcal{SM}} \right),$ where $\left| \mathcal{SM} \right|$ is a topological space, $\mathcal{O}_{\mathcal{SM}}$ is the structure sheaf of $\mathbb{Z}_2$-graded commutative rings and there exists a locally free sheaf $\mathcal{F}$ of rank $n$ such that, for any open subset $\mathcal{U} \subset  \left| \mathcal{SM} \right|,$  
   $ \mathcal{O}_{\mathcal{SM}} \rvert_{\mathcal{U}} \cong \bigwedge^\bullet \mathcal{F}^* \rvert_{\mathcal{U}},$ as sheaves of $\mathbb{Z}_2$-graded rings.
We call
$\mathcal{SM}_{red} = \left( \left| \mathcal{SM} \right| , \mathcal{O}_{\mathcal{SM}_{red}} \right)$, the \emph{reduced manifold},  where $\mathcal{O}_{\mathcal{SM}_{red}} \coloneqq \mathcal{O}_{\mathcal{SM}} / \mathcal{J}$ is the sheaf of functions on $|\mathcal{SM}_{red}|$, with $\mathcal{J}$ the ideal sheaf generated by nilpotent elements. We denote by $i^* \colon \mathcal{O}_{\mathcal{SM}} \to \mathcal{O}_{\mathcal{SM}_{red}}$ the surjective sheaf morphism corresponding to the embedding of locally ringed spaces $i \colon \mathcal{SM}_{red} \hookrightarrow \mathcal{SM}$ and its extension to forms. 

As a consequence of the fact that the structure sheaf $\mathcal{O}_{\mathcal{SM}}$ has both even and odd sections, the de Rham complex of a supermanifold is unbounded from above:
\begin{center}
		\begin{tikzcd}[cells={nodes={minimum height=2em}}]
			0 \arrow[r] & \mathcal{O}_\mathcal{SM} \arrow[r,"\de_{dR}"] &  \mathsf{\Omega}^{(1)} \left( \mathcal{SM} \right) \arrow[r,"\de_{dR}"] &  \ldots \arrow[r,"\de_{dR}"] & \mathsf{\Omega}^{(m)} \left( \mathcal{SM} \right) \arrow[r,"\de_{dR}"] & \ldots
		\end{tikzcd} \ .
	\end{center}

In particular, the complex does not contain the space of integrable sections (i.e., a space of \virgolette top forms'' analogous to the Determinant for an ordinary manifold). Nonetheless, on supermanifolds, a second complex can be introduced, the complex of \emph{integral forms}:
\begin{center}
		\begin{tikzcd}[cells={nodes={minimum height=2em}}]
			\ldots \arrow[r] & \mathsf{\Sigma}^{(0)} \left( \mathcal{SM} \right) \arrow[r,"\delta"] &  \ldots \arrow[r,"\delta"] & \mathsf{\Sigma}^{(m-1)} \left( \mathcal{SM} \right) \arrow[r,"\delta"] & \mathsf{\Sigma}^{(m)} \left( \mathcal{SM} \right) \arrow[r,"\delta"] & 0
		\end{tikzcd} \ ,
	\end{center}
where we denoted $\mathsf{\Sigma}^{(\bullet)} \left( \mathcal{SM} \right) \equiv \mathpzc{B}er \left( \mathcal{SM} \right) \otimes_{\mathcal{O}_{\mathcal{SM}}} S^{m-\bullet} \Pi \mathcal{T} \mathcal{SM}$, i.e., integral forms are defined by tensoring the Berezinian sheaf (the super-analogous of the Determinant sheaf for ordinary manifolds) with (parity-changed) polyvector fields and $\delta$ is the nilpotent operator defined in \cite{Manin:1988ds}. It has been proven that the two complexes are isomorphic in cohomology, see, e.g., \cite{Noja:2021xos}. We shall show that these two complexes fit the following homotopy data
\begin{equation}\label{exhomotopydata}
    \begin{tikzcd}[every arrow/.append style={shift left}]
		\hspace{-.2cm} \left( \mathsf{\Omega} \left( \mathcal{SM} \right) , \de \right) \arrow[loop left, distance=2em, start anchor={[yshift=-0.7ex]west}, end anchor={[yshift=0.7ex]west}, "h_{\mathsf{\Omega}}"] \arrow{r}{Y}  & \left( \mathsf{\Sigma} \left( \mathcal{SM} \right) , \delta \right) \arrow[l, "Z"] \arrow[loop right, distance=2em, start anchor={[yshift=0.7ex]east}, end anchor={[yshift=-0.7ex]east}, "h_{\mathsf{\Sigma}}"]
	\end{tikzcd} \ .
\end{equation}
Explicit realisations of these quasi-isomorphisms are well-known in the Physics community as \emph{Picture Changing Operators} since \cite{Friedan:1985ge}, were formally interpreted in a geometric context in \cite{Belopolsky:1997jz} and translated in the language of (co)chain maps in \cite{Cremonini:2023ccv}. We will focus only on \emph{split} supermanifolds\footnote{For the sake of clarity, one could think about smooth supermanifolds, as all smooth supermanifolds are split (see \cite{Batchelor:1979}).} (i.e., those supermanifolds for which the structure sheaf is globally isomorphic to the reduced structure sheaf times a Grassmann algebra) since in these cases it is possible to find an easy global expression for the map $Z \colon \mathsf{\Sigma} \left( \mathcal{SM} \right) \to \mathsf{\Omega} \left( \mathcal{SM} \right)$, which is surjective on $\mathsf{\Omega} \left( \mathcal{SM}_{red} \right)$. We will address the case of non-split (and non-projected) supermanifolds in \cite{Cremonini-future}. Integral forms are naturally equipped with an $\mathsf{\Omega}(\mathcal{SM})$-module structure, where the action is simply defined by the pairing forms-vector fields. Given an open subset $\mathcal{U} \subseteq \left| \mathcal{SM} \right|$ with local coordinates $\left\lbrace x^i,\theta^\alpha \right\rbrace$ (the former coordinates are even, the latter are odd), the maps $Y$ and $Z$ are defined as follows.
Let $\omega \in \mathsf{\Omega}^{(p)}(\mathcal{SM})$ and consider its local expression
\begin{align}
    \omega \rvert_\mathcal{U} = \sum_{i_1 \ldots i_p = 1}^m \omega_{i_1 \ldots i_p} \left( x \right) \de x^{i_1} \wedge \ldots \wedge \de x^{i_p} + \bar \omega \, ,
\end{align}
where $\bar \omega \in \ker(i^*),$ thus we define
\begin{align}
		 Y^{(p)} & \colon \mathsf{\Omega}^{(p)} \left( \mathcal{U} \right) \to \mathsf{\Sigma}^{(p)} \left( \mathcal{U} \right) \\
	\omega \rvert_{\mathcal{U}} \mapsto Y(\omega \rvert_{\mathcal{U} }) & \coloneqq \sum_{i_1 \ldots i_p = 1}^m \omega_{i_1 \ldots i_p} \left( x \right) \de x^{i_1} \wedge \ldots \wedge \de x^{i_p} \rhd \prod_{\alpha = 1}^n \theta^\alpha \mathcal{D} \otimes \bigwedge_{i=1}^m \pi \partial_{x^i} ,
\end{align}
where $\rhd \colon \mathsf{\Omega} (\mathcal{SM}) \otimes \mathsf{\Sigma} (\mathcal{SM}) \to \mathsf{\Sigma} (\mathcal{SM})$ denotes the action of $\mathsf{\Omega}(\mathcal{SM})$ on $\mathsf{\Sigma}(\mathcal{SM}),$ $\mathcal{D}$ is the local generator of sections of the Berezinian sheaf (see \cite{Manin:1988ds, Noja:2021xos}), and $\pi$ is the parity changing operator. Similarly, let 
$\mu \in \mathsf{\Sigma}^{(m-p)}(\mathcal{SM})$ and consider
\begin{align}
\mu \rvert_{\mathcal{U}} = \sum_{i_1 , \ldots , i_{p} = 1}^m \mu_{i_1 \ldots i_{p}} \left( x , \theta \right) \mathcal{D} \otimes \bigwedge_{s = 1}^p \pi \partial_{x^s} + \bar \mu \, , 
\end{align}
where $\bar \mu$ is the component of the integral form containing at least one vector field $\pi \partial_{\theta},$ and define
\begin{align}
	\nonumber Z^{(m-p)} & \colon \mathsf{\Sigma}^{(m-p)} \left( \mathcal{U} \right) \to \mathsf{\Omega}^{(m-p)} \left( \mathcal{U} \right) \\
    \mu \rvert_{\mathcal{U}} \mapsto Z(\mu \rvert_{\mathcal{U}}) & \coloneqq \sum_{i_1 , \ldots , i_{p} = 1}^m \mu_{i_1 \ldots i_{p}}^{\mathrm{top}} \left( x \right) \iota_{\partial_{x^{i_1}}} \ldots \iota_{\partial_{x^{i_p}}} \left( \de x^1 \wedge \ldots \wedge \de x^m \right) \, ,
\end{align}
where $\iota$ is the contraction by a vector field. 
The global well-definiteness of these maps is shown in \cite{Cremonini:2023ccv}).
The fact that $\Im ( Z ) = \mathsf{\Omega} \left( \mathcal{SM}_{red} \right) \subset \mathsf{\Omega} \left( \mathcal{SM} \right)$ allows us to easily find candidates for the homotopy operators such that Equations \eqref{HTWMSB} and \eqref{HTWMSBA} hold given by (up to coefficients which we neglect for the moment)
\begin{align}
    h_{\mathsf{\Omega}} \coloneqq \iota \chi \ \ , \ \ h_{\mathsf{\Sigma}} \coloneqq \bullet \wedge \pi \chi \ \ , \ \ \chi = \sum_{\alpha} \theta^\alpha \frac{\partial}{\partial \theta^\alpha} \ ,
\end{align}
that is, $h_{\mathsf{\Omega}}$  is the contraction along the \emph{odd Euler vector field}, while $h_{\mathsf{\Sigma}}$ is the (symmetrised) multiplication by the parity-changed field. $\Im ( Z ) = \mathsf{\Omega} \left( \mathcal{SM}_{red} \right)$ also allows to prove that condition (\ref{it:ImZhomo}) of Definition \eqref{def:homotopydatamodule} is satisfied, so that Equation \eqref{exhomotopydata} defines an example of homotopy data with module structure. It can easily be shown that the quasi-isomorphism $Y,Z$ and the homotopy operators $h_{\mathsf{\Omega}}, h_{\mathsf{\Sigma}}$ satisfy the following identities
\begin{align}
    Z \circ Y \circ Z &= Z \, , \qquad h_\mathsf{\Omega} \circ Z = 0 = Z \circ h_\mathsf{\Sigma} \, , \\
    Y \circ Z \circ Y &= Y 
      \, , \qquad h_\mathsf{\Sigma} \circ Y = 0 = Y \circ h_\mathsf{\Omega} \, .
\end{align}
\end{example}

We will discuss the results of this paper for the homotopy data of integral forms on supermanifolds in the upcoming work \cite{Cremonini-future}. Note that this may be particularly relevant since the complex of integral forms does not have any natural algebraic structure.

The last properties of the homotopy data defined by integral forms on supermanifolds set the stage for the introduction of the so-called side conditions for the homotopy operator of quasi-isomorphic cochain complexes. 

\medskip

\subsection{Homotopy Transfer and Side Conditions}~ \label{sect:Motivation}

Let us discuss the interplay between homotopy transfer and the so-called side conditions. Recall from \cite{Markl:2001, Crainic:2004} the following

\begin{defn}[\textbf{Side Conditions}] \label{def:AppA}
 Let $A,B$ be quasi-isomorphic cochain complexes by quasi-isomorphisms $Y, Z$ together with the homotopy $h_B.$ This homotopy $h_B$ is said to satisfy the \emph{side conditions} if the following  hold:
 \begin{align}\label{HMTWSDA}
     h_B \circ Y = 0 \, ,  \quad Z \circ h_B =0  \, , \quad h_B^2=0 \, .
 \end{align}
 An analogous definition holds for the homotopy $h_A,$ i.e., $h_A$ satisfies the side conditions if
 \begin{align}\label{HMTWSDAA}
   Y \circ h_A = 0 \, ,  \quad h_A \circ Z =0  \, , \quad h_A^2=0 \, . 
 \end{align}
 In the following, we will be mainly concerned with the implementation of some side conditions only; in particular, we will use the \emph{right-side condition} for $h_A,h_B$ 
 \begin{align}\label{MasseywuhtCA}
    Z \circ h_B = 0 \, , \quad  \, Y \circ h_A = 0 \, ,
\end{align}
and the \emph{left-side condition} for $h_A, h_B$
\begin{align}\label{HMTWSDAB}
    h_A \circ Z =0 \, , \quad \, h_B \circ Y = 0 \, .
\end{align}
\end{defn}
Recall that, as discussed in \cite{Markl:2001, Markl2006, Crainic:2004}, the side conditions can always be implemented for a deformation retract,\footnote{As discussed in \cite{Crainic:2004}, the deformation retract condition given by $Z \circ Y = 1$. In order to change the homotopy operator $h_B$ such that the side conditions still hold, it suffices to have $Z \circ Y \circ Z = Z$ and $Y \circ Z \circ Y = Y$, as described below.} via an appropriate modification of the homotopy operator.

We point out that the possibility of enforcing the left/right-side conditions, together with the module structure, greatly simplifies the usual homotopy transfer procedure. In particular, it gives an associative algebra on $B$. 

\begin{prop} \label{prop:noAinftyalgebra}
 Let $(A, \de_A)$ and $(B, \de_B)$ be homotopy data with module structure as in Definition \ref{def:homotopydatamodule}, endowed with the homotopy transfer $A_\infty$-algebra $(B, m_n^\hot)$. Moreover, suppose that
$h_A$ satisfies the left-side condition as in Definition \ref{def:AppA}. Then the homotopy transfer $A_\infty$-algebra $(B, m_n^\hot)$ reduces to an associative algebra, i.e., $m_n^\hot = 0$ for $n > 2.$ 
\end{prop}

\begin{proof}
In order to discuss the main argument of the proof, let us consider the case of the 3-product $m_3^\hot$ (see, for instance, \cite{vallette:hal-00757010}):
\begin{align}
    m_3^{\hot} ( b \otimes b' \otimes b'' ) & \coloneqq Y  \bigl( Z (b) \wedge h_A  \left( Z (b') \wedge Z (b'') \right) - h_A  \left( Z (b) \wedge Z (b') \right) \wedge Z (b'') \bigr)  \\
    & = Y  \bigl( Z (b) \wedge (h_A \circ Z) \left( Z(b') \rhd b'' \right) - (h_A \circ Z) \left( Z (b) \rhd b' \right) \wedge Z (b'') \bigr) = 0 \ ,
\end{align}
where we have used the assumption that $Z$ is an $\Im(Z)$-module homomorphism $\wedge \circ \left( Z \otimes Z \right) = Z \circ \rhd  \circ \left( Z \otimes \unit \right)$ and $h_A \circ Z = 0$. 

The above argument may be generalised as follows. Any planar diagram with $n$ entries, of degree $n-2$ (i.e., with an insertion of the homotopy $h_A$ in any internal branch) will be trivially zero, because the homomorphism assumption of $Z$ will always allow moving one of the operators $Z$ acting on the entries to an internal leg, thus leading to $h_A \circ Z,$ see Equation \eqref{MAAA}. 
\end{proof}

Let us discuss under which assumptions some of the side conditions can be implemented on quasi-isomorphic cochain complexes.
We will focus on the implementation of the \emph{right-side condition} \eqref{MasseywuhtCA}, because of its relation to the obstruction to lifting a homotopy equivalence to a \emph{strong homotopy equivalence}, i.e. an extension of the homotopy equivalence given by Equations \eqref{HTWMSB} and \eqref{HTWMSBA} to (higher) $2n$-degree maps $Y^{(2n)}$, $Z^{(2n)}$ and (higher) $2n+1$-degree homotopies $h_A^{(2n+1)},$ $h_B^{(2n+1)},$ for all $n \geq 1,$ see \cite{Markl:2001}. 

In order to explore this relation, let us discuss how the obstruction to the existence of a strong homotopy equivalence can be understood. Note that Equation \eqref{HTWMSB} yields
\begin{align}\label{MasseywuhtD}
      \de_A \circ Z \circ h_B + Z \circ h_B \circ \de_B = \de_A \circ h_A \circ Z + h_A \circ Z \circ \de_B \, .
\end{align}
We can define the usual differential on the space of maps $\partial_{BA} \colon \Hom \left( B , A \right) \to \Hom \left( B , A \right)$  and recast Equation \eqref{MasseywuhtD} as
\begin{align}\label{MasseywuhtDA}
    \partial_{BA} \left( Z \circ h_B \right) = \partial_{BA} \left( h_A \circ Z \right) \ \implies \ \partial_{BA} \left( Z \circ h_B - h_A \circ Z \right) = 0 \ .
\end{align}
The cohomology class
\begin{align} \label{MasseywuhtDAA}
[Z \circ h_B - h_A \circ Z] \in \sfH(\Hom(B,A))
\end{align} 
is the obstruction, described in \cite{Markl:2001}, to lifting a homotopy equivalence to a strong homotopy equivalence\footnote{Note that an analogous argument can be made by considering the combination $Y \circ h_A - h_B \circ Y \in \Hom(A,B)$.}. We may interpret Equation \eqref{MasseywuhtDAA} as an obstruction to simultaneously implementing the left and right-side conditions for $h_A$ and $h_B$ (up to homotopy). In \cite{Markl:2001} it is shown that given a homotopy data, it is always possible to define another one by modifying either the homotopy $h_A$ or the homotopy $h_B$ so that the obstruction is removed, i.e., $Z \circ h_B - h_A \circ Z$ becomes a $\partial_{BA}$-exact element. In our case, we aim to remove the obstruction \eqref{MasseywuhtDAA} while preserving the right-side condition given by Equation \eqref{MasseywuhtCA}; to do so, we recall the modification of the homotopy operator described in \cite{Markl:2001} and show the following property, analogous to the characterisation given in \cite{Crainic:2004}. 

\begin{prop} \label{prop:sideZYZ}
 Let $A, B$ be homotopy data with quasi-isomorphisms $Y,Z$ and suppose that $Z \circ Y \circ Z = Z.$ Then there always exists a homotopy operator $h_B'$ such that 
 $Z \circ h'_B = 0.$  
  Conversely, let the homotopy $h_B$ satisfy Equation \eqref{MasseywuhtCA}, then $Z = Z \circ Y \circ Z$.
\end{prop}
\begin{proof}
  We show that such $h_B'$ always exists by defining (see \cite[Remark 2.3 (i)]{Crainic:2004})
  \begin{align}
      h_B' \coloneqq (\partial_B h_B) \circ h_B \, .
  \end{align}
  Then $h_B'$ satisfies $\unit_B - Y \circ Z =  \partial_B h_B'$ because of the Leibniz rule for $\partial_B$ and $Y \circ Z$ is idempotent because of $Z = Z \circ Y \circ Z$. 

  It follows that
  \begin{align}
      Z \circ h_B' =  Z \circ (\unit_B - Y \circ Z) \circ h_B = 0 \, .
  \end{align}

  Consider now the case where the homotopy $h_B$ satisfies Equation \eqref{MasseywuhtCA}. By acting with $\partial_A$ on Equation \eqref{MasseywuhtCA}, we obtain
  \begin{align}
      0 = \partial_{A} \left( Z \circ h_B \right) = Z \circ \partial_B h_B = Z - Z \circ Y \circ Z \ . 
  \end{align}
\end{proof}

\begin{remark}
    Note that an analogous result holds for the right side condition: if $Y = Y \circ Z \circ Y$, the given homotopy $h_B$ can always be modified in order to obtain the condition $h_B' \circ Y = 0$. In particular, one can always define the homotopy $h_B' =  h_B \circ \partial_B h_B$ which satisfies  $h_B' \circ Y = 0$ and defines a set of homotopy data together with the quasi-isomorphisms $Y,Z$ between the cochain complexes $(A,\de_A)$  and $(B,\de_B)$.
\end{remark}

\begin{prop} \label{prop:RightSideMarkl}
    Let $(A, \de_A),$ $(B, \de_B)$ be quasi-isomorphic cochain complexes with quasi-isomorphisms $Y \in \Hom(A,B),$ $Z \in \Hom(B,A)$ and homotopies $h_A \in \Hom_{-1}(A,A),$ $h_B \in \Hom_{-1}(B,B)$. Suppose that the right-side condition \eqref{MasseywuhtCA} holds. Then there always exists a new homotopy data still satisfying  \eqref{MasseywuhtCA} such that the obstruction \eqref{MasseywuhtDAA} is removed.
\end{prop}

\begin{proof}
 The homotopy $h_A$ can be modified preserving the side condition \eqref{MasseywuhtCA} and removing the obstruction \eqref{MasseywuhtDAA}, i.e.,
\begin{align}\label{MasseywuhtDC}
    h_A \mapsto h_A' = h_A - Z \circ Y \circ h_A = \partial_A h_A \circ h_A \ .
    \end{align}
That gives
    \begin{align}
    Z \circ h_B - h_A' \circ Z =& - h_A' \circ Z = - h_A \circ Z + Z \circ Y \circ h_A \circ Z \\
    =& - \partial_A h_A \circ h_A \circ Z  \\
    =& - h_A \circ \partial_A h_A \circ Z - \partial_{BA} \left( h_A \circ h_A \circ Z \right) \\
    =& - \partial_{BA} \left( h_A \circ h_A \circ Z \right) \, ,
\end{align}
where we used $\left( \partial_{A} h_A \right) \circ Z = \partial_{BA} \left( h_A \circ Z \right) = 0$, as follows from Equation \eqref{MasseywuhtDA}. Note that the homotopy $h'_A$ satisfies Equation \eqref{HTWMSB} because of Proposition \ref{prop:sideZYZ}.
\end{proof}

\begin{remarkl}
    Note that our case of interest shows some peculiarities that do not arise in \cite{Markl:2001}. In particular, by modifying the homotopy $h_B$, the obstruction is removed, but the condition \eqref{MasseywuhtCA} would be lost. This can be seen by considering the redefinition
\begin{align}\label{MasseywuhtDB}
    h_B \mapsto h_B' = h_B - Y \circ Z \circ h_B + Y \circ h_A \circ Z = h_B + Y \circ h_A \circ Z \, ,
 \end{align}   
 which yields
    \begin{align}
     Z \circ h_B' = Z \circ Y \circ h_A \circ Z = h_A \circ Z - \partial_{BA} \left( h_A \circ h_A \circ Z \right) \neq 0 \, .
\end{align}
Note that this redefinition preserves the right-side condition if all the side conditions of Definition \ref{def:AppA} hold, but, in general, it is impossible to modify the homotopy $h_B$ while preserving the right-side condition \eqref{MasseywuhtCA} and removing the obstruction \eqref{MasseywuhtDAA} at the same time. This is due to the fact that even if the side condition \eqref{MasseywuhtCA} is achieved, the obstruction \eqref{MasseywuhtDAA} will be completely determined by the (eventual) cohomology class of $h_A \circ Z,$ which is non-trivial.
\end{remarkl}

Note that a crucial condition in order to redefine the homotopy $h_B$ is given by $Z = Z \circ Y \circ Z$. This condition may be characterised as follows:
 
\begin{lem} \label{lem:ImkerZ}
Let $A, B$ be homotopy data with quasi-isomorphisms $Y,Z$. Then
    \begin{align}
     Z - Z \circ Y \circ Z = 0  \Longrightarrow \Im(Z) \cap \ker(Y) = \{ 0 \}  \ ,
    \end{align}
    and, analogously, $Y - Y \circ Z \circ Y = 0 \Longrightarrow \Im(Y) \cap \ker(Z)=\{ 0\}.$ Moreover, if we assume $Y \circ Z = \left( Y \circ Z \right)^2$ and $Z \circ Y = \left( Z \circ Y \right)^2$ the converse holds as well.
\end{lem}
\begin{proof}
    $(\Rightarrow)$ : By assumption, we have that $Z \circ Y = \unit_A$ on $\Im (Z)$, and let $a \in \Im (Z) \cap \ker (Y)$. We thus have
    \begin{align}
        a = \unit_A a = \left( Z \circ Y \right) ( a ) = 0 \ \implies \ a = 0 \ .
    \end{align}
     The proof for the second part follows the same steps as the one presented above.
     
     $(\Leftarrow)$ : Since $Y \circ Z$ is a projector, it follows that $\left( Z - Z \circ Y \circ Z \right) \left( b \right) \in \ker(Y),$ for all $b \in B.$ On the other hand, $\left( Z - Z \circ Y \circ Z \right) \left( b \right) = \left( Z \circ \left( 1 - Y \circ Z \right) \right) \left( b \right) \in \Im(Z),$ for all $b \in B$. Since $\Im(Z) \cap \ker(Y) = \{ 0 \}$, the statement holds.
\end{proof}

Lastly, we observe that the module structure and the assumption $Z = Z \circ Y \circ Z$ simplify the homotopy transfer procedure as follows:

\begin{prop} \label{prop:closedproducts}
  Let $(A, \de_A)$ and $(B, \de_B)$ be homotopy data with module structure as in Definition \ref{def:homotopydatamodule}, endowed with the homotopy transfer $A_\infty$-algebra $(B, m_n^\hot)$ as in Subsection \ref{subsect:Homotopytransfer}. Suppose that $Z = Z \circ Y \circ Z.$ Then $\partial_B m_n^\hot = 0,$ for all $n \geq 1.$
\end{prop}

\begin{proof}
   The proof will be carried on by induction. Note that $\partial_B m_2^\hot= 0,$ which follows from the Leibniz rule for the associative dga $(A,\de_A, \wedge).$ Furthermore, the higher homotopy transferred products are given by the recursive formula
   \begin{align} \label{MAAA}
       m_{n+1}^\hot (\unit_B^{\otimes (n+1)}) \coloneqq \sum_{s=0}^{n-1} \pm m_n^\hot\bigl( \unit_B^{\otimes s} \otimes ( h_A \circ \wedge \circ Z^{\otimes 2}) \otimes \unit_B^{\otimes (n-s-1)}\bigr) \, , 
   \end{align}
   for $n \geq 2.$ Assume, by induction, that $\partial_B m_n^\hot=0.$ Thus, we have
   \begin{align}
       \partial_B m_{n+1}^\hot (\unit^{\otimes (n+1)}) = & \sum_{s=0}^{n-1} \pm (\partial_B m_n^\hot)\bigl( \unit_B^{\otimes s} \otimes (h_A \circ \wedge \circ Z^{\otimes 2} ) \otimes \unit_B^{\otimes (n-s-1)}\bigr) \\
       & \pm \sum_{s=0}^{n-1} \pm m_n^\hot\bigl( \unit_B^{\otimes s} \otimes \partial_A (h_A \circ \wedge \circ Z^{\otimes 2} ) \otimes \unit_B^{\otimes (n-s-1)}\bigr) \, , \label{MAA}
   \end{align}
   where the first term on the RHS of Equation \eqref{MAA} vanishes because of the induction hypothesis. The second term on the RHS of Equation \eqref{MAA} vanishes as well, as
   \begin{align}
       \partial_A (h_A \circ \wedge \circ Z^{\otimes 2}) = \partial_A (h_A \circ Z \circ \rhd \circ ( Z \otimes \unit_B ) ) = ( \unit_A - Z \circ Y ) \circ Z \circ \rhd \circ ( Z \otimes \unit_B ) = 0 \, ,
   \end{align}
   since $Z \circ Y \circ Z = Z$. Note that the signs in Equation \eqref{MAAA} do not play any role in the proof, since every summand on the RHS of Equation \eqref{MAA} vanishes separately.
\end{proof}

In the following section, we will discuss a novel type of homotopy transfer based on the module structure of Definition \ref{def:homotopydatamodule} involving only the homotopy $h_B$ on $B.$

\medskip

\subsection{Weakening the Side Conditions}~

In this Subsection, we discuss a weaker version of the side conditions, which we refer to as \emph{weak side conditions}. They will be used in the next section to construct $A_\infty$-algebras which can be introduced in a slightly broader regime than those where the SC hold.

\begin{defn}[\textbf{Weak Side Conditions}]\label{def:WSC}
    Let $A,B$ be quasi-isomorphic cochain complexes by quasi-isomorphisms $Y, Z$ together with the homotopy $h_B.$ This homotopy $h_B$ is said to satisfy the \emph{Weak Side Conditions (WSC)} if the following  hold:
    \begin{align}\label{WSCA}
        Y \circ Z \circ h_B = 0 \, , \qquad h_B \circ Y \circ Z = 0 \ .
    \end{align}
\end{defn}
In the previous paragraphs, we have characterised the side conditions in terms of properties of the maps $Y$ and $Z$ in Proposition \ref{prop:sideZYZ}; an analogous result holds here:
\begin{prop} \label{prop:sideZYZY}
 Let $A, B$ be homotopy data with quasi-isomorphisms $Y,Z$ and suppose that $\left( Z \circ Y \right)^2 = Z \circ Y$, i.e., $Z \circ Y$ is a projector. Then there always exists a homotopy operator $h_B'$ such that
 \begin{align}\label{WSCB}
     Y \circ Z \circ h'_B = 0 \, , \qquad h'_B \circ Y \circ Z = 0 \ .
 \end{align} 
  Conversely, let the homotopy $h_B$ satisfy the conditions \eqref{WSCA}, then $Z \circ Y = \left( Z \circ Y \right)^2$.
\end{prop}
\begin{proof}
  The construction of $h_B'$ follows the very same prescription given in \cite[Remark 2.3 (i)]{Crainic:2004}, for the side conditions:
  \begin{align}\label{WSCC}
      h_B' \coloneqq (\partial_B h_B) \circ h_B \circ ( \partial_B h_B) \, .
  \end{align}
  Then $h_B'$ satisfies $\unit_B - Y \circ Z =  \partial_B h_B'$ because of the Leibniz rule for $\partial_B$ and $Y \circ Z$ is idempotent by assumption. 

  It follows that
  \begin{align}\label{WSCD}
      Y \circ Z \circ h_B' & =  Y \circ Z \circ (\unit_B - Y \circ Z) \circ h_B (\unit_B - Y \circ Z) = 0 \, , \\
      h'_B \circ Y \circ Z & = (\unit_B - Y \circ Z) \circ h_B (\unit_B - Y \circ Z) \circ Y \circ Z = 0 \, .
  \end{align}

  Consider now the case where the homotopy $h_B$ satisfies the conditions \eqref{WSCA}. By acting with $\partial_B$ on either of the conditions \eqref{WSCA}, we obtain
  \begin{align}\label{WSCE}
      0 = \partial_{B} \left( Y \circ Z \circ h_B \right) = Y \circ Z \circ \partial_B h_B = Y \circ Z - Y \circ Z \circ Y \circ Z \ , 
  \end{align}
  thus $Y \circ Z$ is idempotent.
\end{proof}
\begin{remark}
    Note that the idempotency of $Y \circ Z$ allows to define a homotopy operator $h'_B$ that satisfies $(h_B')^2 = 0$ as well. In particular, by following the prescription given in \cite{Crainic:2004}, we can introduce
    \begin{align}\label{WSCF}
        h'_B \coloneqq h_B \circ \de_B \circ h_B \ .
    \end{align}
    If $h'_B$ satisfies the conditions \eqref{WSCA}, we have
    \begin{align}\label{WSCG}
        (h_B')^2 & = h_B \circ \de_B \circ h_B^2 \circ \de_B \circ h_B \\
        & = h_B^2 \circ \de_B \circ h_B - h_B^2 \circ \de_B \circ h_B \circ \de_B \circ h_B - h_B \circ Y \circ Z \circ h_B \circ \de_B \circ h_B \\
        & = h_B^2 \circ \de_B \circ h_B - h_B^2 \circ \de_B \circ h_B + h_B^3 \circ Y \circ Z \circ d_B \circ h_B + h_B^3 \circ \de_B^2 \circ h_B = 0 \ ,
    \end{align}
    where we used Equations \eqref{HTWMSBA} and \eqref{WSCA}. We will discuss in more depth the weak side conditions in \cite{Cremonini-future}.
    
    In the next section, the condition $(h_B')^2=0 $  will not be needed in order to construct an $A_\infty$-algebra out of the homotopy data with bimodule structure, thus we will not assume it.
\end{remark}

\section{\texorpdfstring{$A_\infty$}{A-infinity}-Algebras from Module Structures} \label{sect:HomotopyTransMod}

We present in this Section a new homotopy transfer construction that only relies on the homotopy and algebraic data of Definition \ref{def:homotopydatamodule}.
In particular, we
show how the homotopy data with module structure induces a product which is associative up to homotopy. This procedure presents some analogies with the homotopy transfer of an associative algebra. We shall discuss later on how these constructions are related.

\subsection{Homotopy Induced Products from Module Structures} \label{subsect:HMT}~

Starting from the homotopy data with bimodule structure, we can use the $\Im(Z)$-bimodule actions and the quasi-isomorphism $Z \colon A \to B$ to define a product on $B$ inspired by the construction given in \cite{Erler:2013xta, Cremonini:2019xco}. 

\begin{defn} [\textbf{Homotopy Module-Induced Product}]\label{def:HTWMSA}
Let $(A, \de_A)$ and $(B, \de_B)$ be quasi-isomorphic chain complexes endowed with the homotopy data with bimodule structure of Definition \ref{def:homotopydatamodule}. Then we define the  \emph{Homotopy Module-Induced Product (HMIP)}  
\begin{align}
	\nonumber m_2 \colon B^{\otimes 2} \longrightarrow  B
\end{align}
 by
\begin{align} \label{HTWMSG}
    m_2 ( b \otimes b' ) \coloneqq k_1  Z (b)  \rhd b' + k_2 \,  b \lhd Z (b') \ ,
\end{align}

for all $b, b' \in B$ and $k_1,k_2 \in \Bbbk$. We call this procedure \emph{Homotopy Module-Induction (HMI)}.
\end{defn}

The first thing we want to investigate is if there is any relation between the product \eqref{HTWMSG} and the product that one could transfer on $B$ via the usual homotopy transfer, i.e.,
\begin{equation}\label{HTWMSGA}
	m_2^{\hot} \left( b \otimes b' \right) \coloneqq Y \left(  Z (b)  \wedge Z (b')  \right) \, ,  
\end{equation}
for all $b,b' \in B$.

Let us consider the graded vector space $\Hom \left( B^{\otimes \bullet} , B \right)$ and equip it with the differential $\partial_B$ defined as in Equation \eqref{eqn:del}.

In order to discuss the equivalence up to homotopy of the Homotopy Module-Induced Product $m_2$ and the standard homotopy transfer $m_2^\hot,$ we define the map $\tilde{m}_2 \in \Hom \left( B^{\otimes 2} , B \right)$ as
\begin{align} \label{HTWMSGC}
    \tilde{m}_2 \left( b \otimes b' \right) \coloneqq k_1 \left( -1 \right)^{|b|} Z \left(  b \right) \rhd h_B (b') + k_2 \, h_B (b) \lhd Z(b') \, ,
\end{align}
for all $b, b' \in B.$
The map $\tilde{m}_2$ is the key to understanding how the Homotopy Module-Induced Product $m_2$ and the standard homotopy transfer product $m_2^\hot$ are equivalent up to homotopy. 

\begin{lem} \label{lemma:HTWMSG}
In the assumptions of Definition \ref{def:homotopydatamodule}, the Homotopy Module-Induced Product satisfies
\begin{align}\label{HTWMSGE}
	\partial_B \tilde{m}_2 = m_2 - (k_1+k_2) \, m_2^\hot \, , 
\end{align}
i.e., if $k_1 + k_2 \neq 0,$ the Homotopy Module-Induced Product $m_2$ and the (rescaled) homotopy transfer $m_2^\hot$ are equivalent up to homotopy.
\end{lem}
\begin{proof} 
Let us consider the map $\tilde{m}_2$ given by Equation \eqref{HTWMSGC} and compute its differential by applying $\partial_B$. We obtain
\begin{align}\label{HTWMSGD}
    \partial_B \tilde{m}_2 \left( b \otimes b' \right) = m_2 \left( b \otimes b' \right) - k_1 Z \left(  b \right) \rhd  (Y \circ  Z) (b')  - k_2 \left( Y \circ  Z \right) \left( b \right) \lhd Z \left( b' \right) \, ,
\end{align}
for all $b, b' \in B.$
Then Equation \eqref{HTWMSGE} follows straightforwardly from Equation \eqref{HTWMSGD} by using that $Y$ is an $\Im(Z)$-module homomorphism.
\end{proof}

\begin{remark} \label{rmk:HTWMSG} 
If $k_1 + k_2 = 0$ it simply follows that the product $m_2$ is exact. This in particular means that if we restrict on the cohomology $\sfH \left( B \right)$ the product is identically zero; in other words, the restriction to cohomology trivialises the ring structure.
\end{remark}

\subsection{Homotopy Module-Induction and (Higher) Associativity}\label{subsect:HMTA}~

The assumptions of Definition \ref{def:homotopydatamodule} represent the cornerstone of the Homotopy Module-Induction and we shall see that they are crucial to the associativity up to homotopy of the induced product $m_2$. To this, recall that the \emph{Associator} for any product $m_2 \in \Hom(B^{\otimes 2}, B)$ is a trilinear map $\Ass \in \Hom(B^{\otimes 3}, B)$ given by
\begin{align}
    \Ass(b \otimes b' \otimes b'') \coloneqq m_2 \left( m_2 \left( b \otimes b' \right) \otimes b'' \right) - m_2 \left( b \otimes m_2 \left( b' \otimes b'' \right) \right) \, ,
\end{align}
for all $b, b', b'' \in B.$

\begin{prop} \label{prop:AssModule}
Let $(A, \de_A), (B, \de_B)$ be homotopy data with module structure as in Definition \ref{def:homotopydatamodule}. Then the Homotopy Module-Induced Product $m_2 \in \Hom(B^{\otimes 2}, B)$ is associative up to homotopy, i.e. there exists a map $m_3 \in \Hom(B^{\otimes 3}, B)$ such that 
\begin{align} \label{HTWMSGAC}
  \Ass = \partial_B \, m_3 \, ,  
\end{align}
where $\partial_B$ is defined by Equation \eqref{eqn:del}.
\end{prop}

\begin{proof}
The associator for the Homotopy Module-Induced Product $m_2$ in the assumptions of Definition \ref{def:homotopydatamodule} is given by
\begin{align} \label{HTWMSO}
    \Ass(b \otimes b' \otimes b'') = k_1 k_2 \left( Z (b) \wedge Z (b') \right) \rhd b'' - k_1 k_2 \, b \lhd \left( Z (b') \wedge Z (b'') \right) \, ,
\end{align}
for any $b,b',b'' \in B$.

In order to construct a map $m_3 \in \Hom(B^{\otimes 3}, B)$ that satisfies Equation \eqref{HTWMSGAC}, we use the following prescription. 
We obtain the 3-product $m_3$ by inserting the homotopy $h_B$ on the free external branches of $\Ass,$ i.e. the external branches that do not have any input in $Z,$ such as the entries $b''$ in the first term and $b$ in the second term of Equation \eqref{HTWMSO}, respectively.   
We will discuss this prescription in its full generality later on. Hence we have
\begin{align}
    m_3 \left( b \otimes b' \otimes b'' \right) = & k_1 k_2 \left( -1 \right)^{|b| + |b'|} \left( Z (b) \wedge Z (b') \right) \rhd h_B (b'') \label{HTWMSN}  \\ 
    & - k_1 k_2 h_B (b) \lhd \left( Z (b') \wedge Z (b'') \right) \, ,
\end{align}
for all $b,b',b'' \in B$. Note that $m_3$ given by Equation \eqref{HTWMSN} is not unique, since one can always add a closed term with respect to $\partial_B.$ Its differential becomes
\begin{align} \label{HTWMSNA}
      \partial_B m_3 (b \otimes b' \otimes b'') = k_1 k_2 \left( Z (b) \wedge Z (b') \right) \rhd b'' - k_1 k_2 b \lhd \left( Z (b') \wedge Z (b'') \right) \, ,
\end{align}
for all $b, b', b'' \in B.$
 Therefore, by comparing Equations \eqref{HTWMSNA} and \eqref{HTWMSO} we see that Equation \eqref{HTWMSGAC} is verified.
\end{proof}

\begin{remark}
We mentioned above that the product $m_3$ is constructed by suitably inserting the homotopy $h_B$ on a free branch of each term of $\Ass.$ Let us discuss what happens when a different choice, i.e. a different insertion, is made. For instance, consider
\begin{align}
  m_3^{(1, cl)} \left( b \otimes b' \otimes b'' \right) \coloneqq  \left( (Z \circ h_B) (b) \wedge Z (b') 
- (-1)^{|b|} Z (b) \wedge \left( Z \circ h_B \right) (b') \right) \rhd b'' \, .  
\end{align}This term is $\partial_B$-closed, because of the compatibility with the maps $Y$ and $Z$ with the $\Im (Z)$-module action.
In this sense, we follow the prescription used in the proof and we will not consider contributions to $m_3$ given by closed terms.
\end{remark}

\begin{remark}
Let us briefly comment on the relevance of the Property \ref{it:ImZhomo} in Definition \ref{def:homotopydatamodule}. Let us drop assumption \ref{it:ImZhomo}, then the associator of $m_2$ becomes
\begin{align}
    \Ass(b \otimes b' \otimes b'') = &  
	   k_1^2 Z \left( Z \left(  b \right) \rhd b' \right) \rhd b'' + k_1 k_2 Z \left( b \lhd Z \left(  b' \right) \right) \rhd b''\\
        & + k_2^2 b \lhd \left(Z( b') \wedge Z (b'') \right) - k_1^2 \left( Z (b) \wedge Z (b') \right) \rhd b''  \label{HTWMSH} \\
     & - k_1 k_2 b \lhd Z \left( Z \left( b' \right) \rhd b'' \right)  - k_2^2 b \lhd Z \left( b' \lhd Z \left( b'' \right) \right)  \ ,
\end{align}
for all $b, b', b'' \in B.$

By following the same argument as in the proof of Proposition \ref{prop:AssModule}, we define a map $m_3$ by
\begin{align}
    m_3 (b \otimes b' \otimes b'') = & (-1)^{|b|+|b'|} k_1^2 Z \left( Z \left(  b \right) \rhd b' \right) \rhd h_B (b'') \\
    &+ (-1)^{|b|+|b'|} k_1 k_2 Z \left( b \lhd Z \left(  b' \right) \right) \rhd h_B (b'') \label{HTWMSK} \\
        & + k_2^2 h_B (b) \lhd \left(Z( b') \wedge Z (b'') \right) - (-1)^{|b|+|b'|} k_1^2 \left( Z (b) \wedge Z (b') \right) \rhd h_B (b'')   \\
     & - k_1 k_2 h_B (b) \lhd Z \left( Z \left( b' \right) \rhd b'' \right)  - k_2^2 h_B (b) \lhd Z \left( b' \lhd Z \left( b'' \right) \right)  \ ,
\end{align}
for all $b,b',b'' \in B.$  
 
By applying the differential $\partial_B$ to Equation \eqref{HTWMSK}, we obtain
\begin{align}
    \partial_B m_3 \left( b \otimes b' \otimes b'' \right) =& m_2 \left( m_2 \left( b \otimes b' \right) \otimes b'' \right) - m_2 \left( b \otimes m_2 \left( b' \otimes b'' \right) \right)  \\
    & - k_1^2 Z \left( Z \left(  b \right) \rhd b' \right) \rhd (Y \circ Z) (b'') - k_1 k_2 Z \left( b \lhd Z \left(  b' \right) \right) \rhd (Y \circ Z) (b'') \\
    & - k_2^2 (Y \circ Z) (b) \lhd \left(Z( b') \wedge Z (b'') \right) + k_1^2 \left( Z (b) \wedge Z (b') \right) \rhd (Y \circ Z) (b'')  \label{HTWMSL} \\
     & + k_1 k_2 (Y \circ Z) (b) \lhd Z \left( Z \left( b' \right) \rhd b'' \right)  + k_2^2 (Y \circ Z) (b) \lhd Z \left( b' \lhd Z \left( b'' \right) \right)  \ ,
\end{align}
that is, $\partial_B m_3$ is given by the associator of $m_2$ plus six extra terms. It can be verified that it is not possible to modify the definition of the 3-product with other terms constructed with different insertions of the homotopy $h_B$ (or the homotopy $h_A$) that compensate for these extra terms. This means that, in general, an induced product defined as in Equation \eqref{HTWMSG} will not be associative up to homotopy by constructing the map $m_3$ with the prescription discussed in the proof of Proposition \ref{prop:AssModule}. 
\end{remark}

We can now proceed with the construction of the first higher associator.
In particular, we study the \emph{Pentagonator}, denoted by $\mathsf{Pen},$ of the first two products coming from the Homotopy Module-Induction, which is defined by
\begin{align}
 \mathsf{Pen}(b \otimes b' \otimes b'' \otimes b''') \coloneqq & 
    m_2 \left( m_3 \left( b \otimes b' \otimes b'' \right) \otimes  b''' \right) + \left( -1 \right)^{|b|} m_2 \left( b \otimes m_3 \left( b' \otimes b'' \otimes b''' \right) \right)  \\
	\nonumber & - m_3 \left( m_2 \left( b \otimes b' \right) \otimes b'' \otimes b''' \right) + m_3 \left( b \otimes m_2 \left( b' \otimes b'' \right) \otimes b''' \right) \\
 &- m_3 \left( b \otimes b' \otimes m_2 \left( b'' \otimes b''' \right) \right) \, ,
\end{align}
for all $b,b',b'',b'''\in B.$ 
By using Equations \eqref{HTWMSG}, \eqref{HTWMSN} and Definition \ref{def:homotopydatamodule}, we obtain
\begin{align}
	\label{HTWMSQ} \Pen \left( b \otimes b' \otimes b'' \otimes b''' \right) = & k_1 k_2^2 \left( -1 \right)^{|b| + |b'|} \left( Z (b) \wedge Z (b') \right) \rhd h_B ( b'') \lhd Z (b''')  \\ 
    &- k_1 k_2^2 \, h_B ( b ) \lhd \left( Z (b') \wedge Z (b'') \wedge Z (b''') \right) \\
    & + k_1^2 k_2 \left( -1 \right)^{|b|+|b'|+|b''|} \left( Z (b) \wedge Z (b') \wedge Z (b'') \right) \rhd h_B (b''') \\
    & - k_1^2 k_2 \left( -1 \right)^{|b|} Z (b) \rhd h_B (b') \lhd \left( Z (b'') \wedge Z (b''') \right) \\
    & + k_1^2 k_2 \left( -1 \right)^{|b|+|b'|} \left( Z (b) \wedge Z (b') \wedge  \left(Z \circ h_B  \right) (b'') \right) \rhd b''' \\
    &- k_1^2 k_2 \left( (Z \circ  h_B) (b)  \wedge Z (b') \wedge Z (b'') \right) \rhd b''' \\
    & + k_1 k_2^2 \left( -1 \right)^{|b| + |b'|+|b''|} b \lhd \left( Z (b') \wedge Z (b'') \wedge (Z \circ h_B) (b''') \right) \\
    \nonumber 
    & - k_1 k_2^2 \left( -1 \right)^{|b|} b \lhd \left( (Z \circ h_B) (b') \wedge Z (b'') \wedge Z (b''') \right) \\
	 & + k_1^2 k_2 \, h_B \left( Z (b) \rhd b' \right) \lhd \left( Z (b'') \wedge Z (b''') \right) \\ 
    &  + k_1 k_2^2 \, h_B \left( b \lhd Z (b') \right) \lhd \left( Z (b'') \wedge Z (b''') \right) \\
    & - k_1^2 k_2 \left( -1 \right)^{|b|+|b'|} \left( Z (b) \wedge Z (b') \right) \rhd h_B \left( Z (b'') \rhd b''' \right) \\
    &- k_1 k_2^2 \left( -1 \right)^{|b|+|b'|} \left( Z (b) \wedge Z (b') \right) \rhd h_B \left( b'' \lhd Z (b''') \right) \ ,
\end{align}
for all $b,b',b'',b''' \in B.$
The product $m_4 \in \Hom \left( B^{ \otimes 4} , B \right)$ can be constructed, out of the expression of the pentagonator, by following the same prescription discussed in the proof of Proposition \ref{prop:AssModule}: whenever there is a \virgolette naked'' external branch, i.e., an external branch without insertions of $Z$ or $h_B$, apply the homotopy operator $h_B$\footnote{Of course, one should also fix the sign depending on which inputs and how many homotopy operators have been crossed when applying $h_B$.}. If a term presents no naked branches, it will be neglected. In particular, out of the twelve terms of Equation \eqref{HTWMSQ}, we can construct the product $m_4$ as
\begin{align}
	\label{HTWMSR} m_4 \left( b \otimes b'\otimes b'' \otimes b''' \right) \coloneqq & - (-1)^{|b''|} k_1^2 k_2 \left( Z (b) \wedge Z (b') \wedge  \left(Z \circ h_B  \right) (b'') \right) \rhd h_B (b''') \\
    & + (-1)^{|b|+|b'|+|b''|} k_1^2 k_2 \left( (Z \circ  h_B) (b)  \wedge Z (b') \wedge Z (b'') \right) \rhd h_B (b''') \\
    & + k_1 k_2^2 \left( -1 \right)^{|b| + |b'|+|b''|} h_B (b) \lhd \left( Z (b') \wedge Z (b'') \wedge (Z \circ h_B) (b''') \right) \\
    \nonumber 
    & - k_1 k_2^2 \left( -1 \right)^{|b|} h_B (b) \lhd \left( (Z \circ h_B) (b') \wedge Z (b'') \wedge Z (b''') \right) \\
	 & - k_1^2 k_2 \, h_B \left( Z (b) \rhd h_B (b') \right) \lhd \left( Z (b'') \wedge Z (b''') \right) \\ 
    &  - k_1 k_2^2 \, h_B \left( h_B (b) \lhd Z (b') \right) \lhd \left( Z (b'') \wedge Z (b''') \right) \\
    & + k_1^2 k_2 \left( -1 \right)^{|b''|} \left( Z (b) \wedge Z (b') \right) \rhd h_B \left( Z (b'') \rhd h_B (b''') \right) \\
    & + k_1 k_2^2 \left( -1 \right)^{|b|+|b'|} \left( Z (b) \wedge Z (b') \right) \rhd h_B \left( h_B (b'') \lhd Z (b''') \right) \, ,
\end{align}
for all $b,b',b'',b''' \in B,$ where only eight terms are involved in the construction.

In order to have the map $m_4$ satisfying the higher associativity relation for $\Pen,$ we can either constrain the homotopy data or fix a relation between the coefficients $k_1$ and $k_2$ appearing in the product $m_2,$ Equation \eqref{HTWMSG}. This can be stated in the following
\begin{lem} \label{lem:Penm4}
The product $m_4$ satisfies the equation
\begin{align} \label{HTWMSU}
    \mathsf{Pen} = \partial_B m_4
\end{align}
if either of the following conditions hold
\begin{enumerate}
    \item $k_1 \vee k_2 = 0$; in this case, the map $m_2$ becomes associative, the maps $m_3$ and $m_4$ are zero, i.e. Equation \eqref{HTWMSU} is trivially   satisfied; \label{it:Pen1}
    \item $k_1 = - k_2$; in this case, the non-associative map $m_2$ becomes $\partial_B$-exact, as it is shown in Equation \eqref{HTWMSGE}; \label{it:Pen2}
    \item the homotopy $h_B$ satisfies the weak side conditions, see Definition \ref{def:WSC}. \label{it:Pen3}
\end{enumerate}
\end{lem}
\begin{proof}
 Equation \eqref{HTWMSR} satisfies
\begin{equation} \label{HTWMSS}
    \partial_B m_4 = \Pen + \mathfrak{R}_Y \ ,
\end{equation}
where $\mathfrak{R}_Y$ is a collection of extra terms depending explicitly on the map $Y$ given by
\begin{align}
     \mathfrak{R}_Y \left( b \otimes b'\otimes b'' \otimes b''' \right) =& - k_1 k_2 \left( k_1 + k_2 \right) Y \left( Z \left( b \right) \wedge Z \left( b' \right) \wedge \left( Z \circ h_B \right) \left( b'' \right) \wedge Z \left( b''' \right) \right) \\
    & + k_1 k_2 \left( k_1 + k_2 \right) Y \left( \left( Z \circ h_B \right) \left( b \right) \wedge Z \left( b' \right) \wedge Z \left( b'' \right) \wedge Z \left( b''' \right) \right) \\
    & - k_1 k_2 \left( k_1 + k_2 \right) Y \left( Z \left( b \right) \wedge Z \left( b' \right) \wedge Z \left( b'' \right) \wedge \left( Z \circ h_B \right) \left( b''' \right) \right) \\
    & \label{HTWMST} + k_1 k_2 \left( k_1 + k_2 \right) Y \left( Z \left( b \right) \wedge \left( Z \circ h_B \right) \left( b' \right) \wedge Z \left( b'' \right) \wedge Z \left( b''' \right) \right)  \\
    & - k_1 k_2 \left( k_1 + k_2 \right) \left( h_B \circ Y \right) \left( Z \left( b \right) \wedge Z \left( b' \right) \right) \lhd \left( Z \left( b'' \right) \wedge Z \left( b''' \right) \right) \\
    & + k_1 k_2 \left( k_1 + k_2 \right) \left( Z \left( b \right) \wedge Z \left( b' \right) \right) \rhd \left( h_B \circ Y \right) \left( Z \left( b'' \right) \wedge Z \left( b''' \right) \right) \, ,
\end{align}
for all $b, b', b'', b''' \in B.$
    If any of the conditions above hold, Equation \eqref{HTWMST} vanishes.
\end{proof}

\begin{remark} \label{rmk:lemma10sd}
    Note that the weak side conditions of Definition \ref{def:WSC} hold true whenever the side conditions on Definition \ref{def:AppA} are satisfied.
\end{remark}

In order to show that the Homotopy Module-Induction yields an $A_\infty$-algebra, consider the $A_\infty$-relations
\begin{align}\label{HMTWSDB}
    \partial_B m_n = \Ass_n \coloneqq \sum_{i=2}^{n-1} (-1)^{i} m_i \circ m_{n-i+1} \, ,
\end{align}
for all $n \in \mathbb{N}_+,$ where on the RHS we define
\begin{align} \label{HMTWSDBA}
    m_i \circ m_k \bigl(\unit_B^{\otimes (i+k-1)} \bigr) \coloneqq \sum_{j=0}^{i-1} (-1)^{j+i+ k(i-j-1)} m_i \bigl(\unit^{\otimes j}_B \otimes m_k (\unit_B^{\otimes k})\otimes \unit_B^{\otimes (i-j-1)} \bigr) \, .
\end{align}

Lemma \ref{lem:Penm4} shows under which assumptions we may obtain an $A_\infty$-algebra from the Homotopy Module-Induction;\footnote{We will neglect the situation described in case \ref{it:Pen1}.} before exploring the construction under these assumptions, we show that a simpler $A_\infty$-algebra can be obtained when the homotopy data satisfies the side conditions (see Definition \ref{def:AppA}), as a special case of Lemma \ref{lem:Penm4}, see Remark \ref{rmk:lemma10sd}.  

\subsection{\texorpdfstring{$A_\infty$}{A-infinity}-Algebras from Homotopy Module-Induction with Side Conditions}~
  
We shall show that the Homotopy Module-Induction yields an $A_\infty$-algebra when the homotopy data with bimodule structure satisfies the side conditions.

In particular, by enforcing the left and right side conditions for the homotopy $h_B$, as in Definition \ref{def:AppA}, we immediately see that the product $m_4$ defined by Equation \eqref{HTWMSR} simplifies to 
\begin{align}
	\label{HTWMSV} m_4 \left( b \otimes b' \otimes b'' \otimes b''' \right) \coloneqq & - k_1^2 k_2 h_B \left( Z (b) \rhd h_B (b') \right) \lhd \left( Z (b'') \wedge Z (b''') \right) \\ 
    &  - k_1 k_2^2 h_B \left( h_B (b) \lhd Z (b') \right) \lhd \left( Z (b'') \wedge Z (b''') \right) \\
    & + k_1^2 k_2 \left( -1 \right)^{|b''|} \left( Z (b) \wedge Z (b') \right) \rhd h_B \left( Z (b'') \rhd h_B (b''') \right) \\
    & + k_1 k_2^2 \left( -1 \right)^{|b|+|b'|} \left( Z (b) \wedge Z (b') \right) \rhd h_B \left( h_B (b'') \lhd Z (b''') \right) \ ,
\end{align}
for all $b,b',b'',b''' \in B,$
and that Equation \eqref{HTWMST} vanishes with no further assumptions on the coefficients $k_1, k_2,$ i.e. Equation \eqref{HTWMSU} holds and the pentagonator of $m_3$ and $m_2$ vanishes up to homotopy. 

Note that, as stated in Remark \ref{rmk:lemma10sd}, $m_4$ given by Equation \eqref{HTWMSV} satisfies Lemma \ref{lem:Penm4}. We can thus show that the following $A_\infty$-algebra structure can be induced on $B$:

\begin{thm}[\textbf{Homotopy Module-Induction with SC}] \label{thm:Ainftyside}
 Let $(A, \de_A), (B, \de_B)$ be homotopy data with module structure as in Definition \ref{def:homotopydatamodule} such that the homotopy $h_B$ satisfies the side conditions of Definition \ref{def:AppA}. 
Then the Homotopy Module-Induction gives an $A_\infty$-algebra $(B, m_n)$ with $m_1= \de_B,$ $m_2$ given by Equation \eqref{HTWMSG}, and higher products
 \begin{align}
      m_n \coloneqq (m_{n-1} \circ m_2) \circ h_B \ , \qquad \text{for } n \geq 3 \ . \label{HMTWSDBD}
 \end{align}
\end{thm}
\begin{remark} \label{rmk:constructingnprod}
    Note that the definition of $n$-product given by Equation \eqref{HMTWSDBD} reflects the prescription described in the previous paragraphs: to define the $n$-ary product $m_n$, start from the expression of the $n$-agonator $\Ass_n = \sum_{i=2}^{n-1} (-1)^i m_i \circ m_{n-i+1}$ and compose (with proper signs) the homotopy operator $h_B$ on the resulting free inputs. Because of the properties of the map $Z$ and the side conditions, we have that the only contributions with a free input come from $m_{n-1} \circ m_2$. Note that terms in $m_2 \circ m_{n-1}$ could have a free input, but they vanish because of the side conditions. In particular, the only terms having a free input are given by the composition of $m_{n-1}$ on the input of $m_2$ presenting $Z$. Since $Z$ is an $\Im (Z)$-module homomorphism, we can always move it in the resulting expression until, eventually, it is composed with a homotopy operator $h_B$, thus giving zero because of the side conditions. 
\end{remark}

\begin{proof}
    We shall start by considering any $n$-agonator $\Ass_n.$ The building blocks of the $n$-product $m_n$ are the terms of $\Ass_n$ presenting one free input (or one free external branch, in the language of tree diagrams), as discussed in Remark \ref{rmk:constructingnprod}. 
   There will be $2^{n-1}$ of such terms for $\Ass_n,$ with $n \geq 3.$ Note that, from the second equality in Equation \eqref{HMTWSDB}, it is impossible to have terms in $\Ass_n$ with more than one free input.   

    We use these $2^{n-1}$ terms to construct the $n$-product by inserting the homotopy operator $h_B$ in the free input. Because of the right side condition, the terms appearing in the product $m_n$ are $2^{n-2}.$ 
    Note that all the terms contributing to the $n$-product will present the $h_B$ operator on connected branches starting from an upper external one and descending to the innermost one.

    The calculation of $\partial_B m_n,$ where $m_n$ is constructed with the prescription discussed above, satisfies Equation \eqref{HMTWSDB}. Note that by applying $\partial_B$ to each term of $m_n,$ we will obtain $n-2$ new terms simply by replacing each one of the operators $h_B$ with $\unit_B$ (each term might pick up a sign based on how many operators $h_B$ are in the diagram, starting from the innermost one, until reaching the one substituted by the identity). There will be no other contributions, since $\partial_B$ can only act on $h_B$ and the terms in which $Y \circ Z$ appear will vanish because of the side conditions and Assumption \ref{it:ImZhomo} in Definition \ref{def:homotopydatamodule}.

The constructive argument to prove this Theorem leads to an abstract proof based on the properties of the $n$-agonators with side conditions. From now on we assume that all the side conditions on $h_B$ hold.
    We start from 
    \begin{align}\label{dimHMIWSCA}
        \partial_B \Ass_n = 0 \, ,
    \end{align}
    which encodes the following relations 
    \begin{align}\label{dimHMIWSCB}
        \partial_B \Ass_n = &\partial_B \left( \sum_{i=2}^{n-1} (-1)^i m_i \circ m_{n-i+1} \right) \\
         = &\sum_{i=2}^{n-1} \left( \left( \sum_{j=2}^{i-1} (-1)^{i+j} m_j \circ m_{i-j+1} \right) \circ m_{n-i+1} \right. \\
        & + m_i \circ \left( \sum_{j=2}^{n-i} (-1)^j m_j \circ m_{n-i-j+2} \right)  \Bigg) = 0 \ .
    \end{align}
    We can recast it as 
    \begin{align}
        \sum_{j=2}^{n-2} (-1)^j \left( m_j \circ m_{n-j} \right) \circ m_2  = & (-1)^n \sum_{i=2}^{n-2} \left[ \left( \sum_{j=2}^{i-1} (-1)^{i+j} m_j \circ m_{i-j+1} \right) \circ m_{n-i+1} \right. \\
        & + m_i \circ \left( \sum_{j=2}^{n-i} (-1)^j m_j \circ m_{n-i-j+2} \right)  \Bigg] \, . \label{dimHMIWSCC}
    \end{align}
    In order to prove the Theorem, we will proceed by induction (the first inductive steps have been verified explicitly in the previous paragraphs). Let us recall the definition of the higher products given by Equation \eqref{HMTWSDBD}, i.e.,
    \begin{align}\label{dimHMIWSCD}
        m_n \coloneqq (-1)^{n-1} \Ass_n \circ h_B = \left( m_{n-1} \circ m_2 \right) \circ h_B \ ,
    \end{align}
    for all $n\geq 3$, where, from the first to the second equality, we used $Z \circ h_B = 0$ and $h_B^2=0$. Thus, by acting with the differential $\partial_B$, we have
    \begin{align}\label{dimHMIWSCE}
        \partial_B m_n = \left( \partial_B m_{n-1} \circ m_2 \right) \circ h_B + (-1)^{n-1} m_{n-1} \circ m_2 \circ \left( \unit_B - Y \circ Z \right) \ .
    \end{align}
    The term given by $m_{n-1} \circ m_2 \circ Y \circ Z$ vanishes because of the Property \ref{it:ImZhomo} of Definition \ref{def:homotopydatamodule}, i.e., $Y, \, Z$ are $\Im(Z)$-bimodule homomorphisms, this allows to obtain the composition $h_B \circ Y =0$ with the first of the nested homotopy operator $h_B$ appearing in any $m_{n-1}$, thus giving zero\footnote{This is true for all $n \neq 3$. The case of $\Ass_3$ has been verified separately.} because of the side conditions. By inductive hypothesis we have $\partial_B m_i = \sum_{j=2}^{i-1} m_j \circ m_{i-j+1}$,  for all $ i \leq n-1$, so that the first term reads, by using Equation \eqref{dimHMIWSCC},
    \begin{align}
        \left( \sum_{i=2}^{n-2} (-1)^i m_i \circ m_{n-i} \right) \circ m_2 \circ h_B = & (-1)^n \sum_{i=2}^{n-2} \left[ \left( \sum_{j=2}^{i-1} (-1)^{i+j} m_j \circ m_{i-j+1} \right) \circ m_{n-i+1} \right. \\
        & + m_i \circ \left( \sum_{j=2}^{n-i} (-1)^j m_j \circ m_{n-i-j+2} \right)  \Bigg] \circ h_B \ . \label{dimHMIWSCF}
    \end{align}
    The first term on the RHS of Equation \eqref{dimHMIWSCF} vanishes, since $m_{n-i+1} \neq m_2,$ as $i \leq n-2$, hence the composition with $h_B$ is zero because of the side conditions. 
    
    The second term on the RHS of Equation \eqref{dimHMIWSCF} has only non-zero term given by the one with $m_2$ composed on the right,  i.e., the term with $j=n-i$ in the sum, whereas all the other terms vanish because of the side conditions. Equation \eqref{dimHMIWSCF}, together with the second term of Equation \eqref{dimHMIWSCE}, leads to
    \begin{align}
        \partial_B m_n = \sum_{i=2}^{n-2} (-1)^i m_i \circ m_{n-i+1} + (-1)^{n-1} m_{n-1} \circ m_2 = \sum_{i=2}^{n-1} (-1)^i m_i \circ m_{n-i+1} = \Ass_n \, ,
    \end{align}
    thus completing the proof.
\end{proof}

\subsection{General Construction of Homotopy Module-Induced \texorpdfstring{$A_\infty$}{A-infinity}-Algebras}~

In this subsection, we construct an $A_\infty$-algebra via HMI in different cases than Theorem \ref{thm:Ainftyside}. In particular, we will prove an analogous theorem that holds both when the weak side conditions are assumed and in the case where we do not assume any condition on the maps involved, but we fix the coefficients of the product $m_2$ to satisfy $k_1 = -k_2$. These are the relevant cases \ref{it:Pen2} and \ref{it:Pen3} of Lemma \ref{lem:Penm4}.

Before proceeding with the construction, we describe some notations that will be used in the following. Given a $n$-ary map $\varphi \in \Hom \left( B^{\otimes n} , B \right)$, we denote
\begin{align}\label{notationcompositionsA}
    \varphi \circ h_{B^{\otimes n}} = \varphi \circ (h_B \otimes \unit_B^{\otimes n-1} + \unit_B \otimes h_B \otimes \unit_B^{\otimes n-2} + \ldots + \unit_B^{\otimes n-1} \otimes h_B) \equiv \varphi \circ h_B
\end{align}
as the $n$-ary map where we inserted once the homotopy operator $h_B$ on the inputs of $\varphi$. We denote
\begin{align}
    \varphi \circ_{fb} h_B
\end{align}
as the $n$-ary map where $h_B$ has been composed in the free inputs (\virgolette free branches'' in the graphical representation of maps as planar trees) of $\varphi$ only. For example, given the definitions in Equations \eqref{HTWMSG} and \eqref{HTWMSGC} one sees that
\begin{align}\label{notationcompositionsB}
    \tilde{m}_2 \coloneqq m_2 \circ_{fb} h_B \ .
\end{align}
Note that the maps we will use will always have at most one free input. Given $m$ and $n$-ary maps $\varphi \in \Hom \left( B^{\otimes m} , B \right)$ and $\psi \in \Hom \left( B^{\otimes n} , B \right)$, we denote with
\begin{align}
    \varphi \circ_Z \psi \quad \text{and} \quad \varphi \circ_{h_B} \psi
\end{align}
the $m+n-1$-ary maps where $\psi$ has only been composed in the inputs with $Z$, respectively $h_B$, of $\varphi$. For example, consider $\varphi ( b \otimes b') = \psi ( b \otimes b') \coloneqq b \lhd Z (b')$, for all $b,b' \in B$, we have
\begin{align}
    (\varphi \circ_Z \psi) ( b \otimes b' \otimes b'') = - b \lhd \left( Z \left( b' \lhd Z (b'') \right) \right) \ ,
\end{align}
for all $b,b',b'' \in B$, where the minus sign comes from the prescription of Equation \eqref{HMTWSDBA}.

\begin{thm}[\textbf{Homotopy Module-Induced $A_\infty$-Algebra}] \label{thm:weaksideorcoefficients}
   Let $(A, \de_A), (B, \de_B)$ be homotopy data with module structure as in Definition \ref{def:homotopydatamodule} and assume that either of the following conditions hold
   \begin{enumerate}
       \item the homotopy $h_B$ satisfies the WSC of Definition \ref{def:WSC};
       \item the coefficients appearing in Equation \eqref{HTWMSG} are such that $k_1 = - k_2.$ 
   \end{enumerate}
   Then the Homotopy Module-Induction gives the $A_\infty$-algebras $(B, m_n)$ with $m_1 \coloneqq \de_B, $ $m_2$ is defined by Equation \eqref{HTWMSG} and higher products\footnote{For $n=3$ only the second equality should be considered.\label{Footnoteonm_3}} 
   \begin{align}\label{higherproductsinrelaxedtheorem}
       m_n  \coloneqq & (-1)^{n-1} \left( m_2 \circ m_{n-1} \right) \circ_{fb} h_B + \left( m_{n-1} \circ m_2 \right) \circ_{fb} h_B \\
         = & \tilde{m}_2 \circ_Z m_{n-1} + m_{n-1} \circ_{h_B} \tilde{m}_2 \, , \label{eqn:generalmn}
   \end{align}
   where $\tilde{m}_2$ is given by Equation \eqref{HTWMSGC}.
\end{thm}

\begin{proof}
    We follow a similar strategy to the proof of Theorem \ref{thm:Ainftyside}. We start with the relation
    \begin{align}
        \partial_B \Ass_n = 0 \ ,
    \end{align}
    which yields
    \begin{align}\label{associativitycomposition}
        \partial_B \Ass_n  =& \partial_B \left( \sum_{i=2}^{n-1} (-1)^i m_i \circ m_{n-i+1} \right) \\
         =& \sum_{i=2}^{n-1} \left( \left( \sum_{j=2}^{i-1} (-1)^{i+j} m_j \circ m_{i-j+1} \right) \circ m_{n-i+1} \right. \\
        & + m_i \circ \left( \sum_{j=2}^{n-i} (-1)^j m_j \circ m_{n-i-j+2} \right)  \Bigg) = 0 \, .
    \end{align}
    We recast Equation \eqref{associativitycomposition} as
    \begin{align}
        \sum_{j=2}^{n-2} &(-1)^j \left( m_j \circ m_{n-j} \right) \circ m_2 + (-1)^{n-1} m_2 \circ \sum_{j=2} (-1)^j m_j \circ m_{n-i-j+2} = \\
         &(-1)^n \sum_{i=2}^{n-2} \left[ \left( \sum_{j=2}^{i-1} (-1)^{i+j} m_j \circ m_{i-j+1} \right) \circ m_{n-i+1} \right]  \\
         &+ (-1)^n \sum_{i=3}^{n-2} \left[ m_i \circ \left( \sum_{j=2}^{n-i} (-1)^j m_j \circ m_{n-i-j+2} \right) \right] \ .
    \end{align}
    We now proceed by induction (the first inductive steps have been proved in the previous paragraphs) and define (as stated in Footnote \ref{Footnoteonm_3}, for $n=3$ only the second expression should be considered)
    \begin{align}\label{definitionproductmn}
        m_n  \coloneqq & (-1)^{n-1} \left( m_2 \circ m_{n-1} \right) \circ_{fb} h_B + \left( m_{n-1} \circ m_2 \right) \circ_{fb} h_B \\
         = & \tilde{m}_2 \circ_Z m_{n-1} + m_{n-1} \circ_{h_B} \tilde{m}_2 \ ,
    \end{align}
    where $\circ_Z$ and $\circ_{h_B}$ denote the composition along an input with $Z$ or $h_B$, respectively, as introduced above. This means we are enforcing the prescription of inserting a homotopy operator on every term with a free external input of the associator, neglecting the other terms. By using the inductive hypothesis $\partial_B m_i = \sum_{i=2}^{i-1} (-1)^j m_j \circ m_{i-j+1}$ and $\partial_B \tilde{m}_2 = m_2 - (k_1 + k_2) m_2^\hot$ (see Lemma \ref{lemma:HTWMSG}), we can compute its differential as
    \begin{align}\label{nomeprovvisorio1}
        \partial_B m_n  = & \left( m_2 - (k_1 + k_2) m_2^\hot \right) \circ_Z m_{n-1} - \tilde{m}_2 \circ_Z \sum_{i=2}^{n-2} (-1)^i m_i \circ m_{n-i} \\
        &+ \left( \sum_{i=2}^{n-2} (-1)^i m_i \circ m_{n-i} \right) \circ_{h_B} \tilde{m}_2 
        + (-1)^{n-1} m_{n-1} \circ_{h_B} \left( m_2 - (k_1 + k_2) m_2^\hot \right) \ .
    \end{align}
    Let us first discuss the terms containing $m_2^\hot$ in Equation \eqref{nomeprovvisorio1}: on the one hand, if condition \ref{it:mainthm2} of the theorem holds, these terms have vanishing coefficients. On the other hand, if condition \ref{it:mainthm1} of the theorem holds, by using the Property \ref{it:ImZhomo} of the maps $Y$ and $Z$ in Definition \ref{def:homotopydatamodule}, the composition $Y \circ Z$ can be moved until it composes with a homotopy operator $h_B$\footnote{This is true for $n>3$. The case $n=3$ has been verified explicitly in the previous section.}, thus giving zero because of the weak side conditions \eqref{WSCA}. Let us now focus on the second term on the RHS of Equation \eqref{nomeprovvisorio1}, which gives (by summing and subtracting the same terms)
    \begin{align}\label{nomeprovvisorio2}
        - \tilde{m}_2 \circ_Z \sum_{i=2}^{n-2} (-1)^i m_i \circ m_{n-i}  =& - (-1)^{n-2} \left[ m_2 \circ \left( \sum_{i=2}^{n-1} (-1)^i m_i \circ m_{n-i} \right) \right] \circ_{fb} h_B  \\
        & + (-1)^{n-2} \left[ m_2 \circ_{h_B} \left( \sum_{i=2}^{n-1} (-1)^i m_i \circ  m_{n-i} \right) \right] \circ_{fb} h_B \ .
    \end{align}
    Here there is a slight abuse of notation: when writing $m_2 \circ_{h_B} \ldots$, we actually mean the composition on the free input of $m_2$, since $m_2$ does not contain any $h_B$. This notation comes from the fact that this $m_2$ originated from $\tilde{m}_2$ which instead contains $h_B$. The second term on the RHS of Equation \eqref{nomeprovvisorio2} yields
    \begin{align} \label{eqn:secondtermsecond}
        (-1)^{n-2} \left[ m_2 \circ_{h_B} \left( \sum_{i=2}^{n-1} (-1)^i m_i \circ  m_{n-i} \right) \right] \circ_{fb} h_B &= \\
       \left[ (-1)^{n-2} m_2 \circ_{h_B} \left( m_2 \circ m_{n-2} + m_{n-2} \circ m_2 \right) \right] \circ_{fb} h_B &= m_2 \circ_{h_B} m_{n-1} \ ,
    \end{align}
    where the non-trivial summands are given by $i=2,n-2$ and we used the definition of $m_{n-1}$. Note that the other terms in the sum vanish because of the free branch composition. Equation \eqref{eqn:secondtermsecond}, together with the first term on the RHS of Equation \eqref{nomeprovvisorio1}, gives
    \begin{align} \label{eqn:result1}
        m_2 \circ_Z m_{n-1} + m_2 \circ_{h_B} m_{n-1} = m_2 \circ m_{n-1} \ .
    \end{align}
    A similar argument holds for the third term on the RHS of Equation \eqref{nomeprovvisorio1}, leading to
    \begin{align}\label{nomeprovvisorio3}
        \left( \sum_{i=2}^{n-2} (-1)^i m_i \circ m_{n-i} \right) \circ_{h_B} \tilde{m}_2  = & \left[ \left( \sum_{i=2}^{n-2} (-1)^i m_i \circ m_{n-i} \right) \circ m_2 \right] \circ_{fb} h_B \\
        & - \left[ \left( \sum_{i=2}^{n-2} (-1)^i m_i \circ m_{n-i} \right) \circ_Z m_2 \right] \circ_{fb} h_B \, .
    \end{align}
    For the second term on the RHS of Equation \ref{nomeprovvisorio3}, we have
    \begin{align} \label{eqn:secondtermthird}
- \left[ \left( \sum_{i=2}^{n-2} (-1)^i m_i \circ m_{n-i} \right) \circ_Z m_2 \right] \circ_{fb} h_B &=  \\    
- \left[ \left( m_2 \circ m_{n-2} + (-1)^{n-2} m_{n-2} \circ m_2 \right) \circ_Z m_2 \right] \circ_{fb} h_B &= (-1)^{n-1} m_{n-1} \circ_Z m_2 \, .
    \end{align}
    Equation \eqref{eqn:secondtermthird}, together with the last one on the RHS of Equation \eqref{nomeprovvisorio1}, yields 
    \begin{align} \label{eqn:result2}
        (-1)^{n-1} m_{n-1} \circ_Z m_2 + (-1)^{n-1} m_{n-1} \circ_{h_B} m_2 = (-1)^{n-1} m_{n-1} \circ m_2 \, .
    \end{align}
    Let us now consider the remaining terms from Equations \eqref{nomeprovvisorio2} and \eqref{nomeprovvisorio3}, that is
    \begin{align}
        (-1)^{n-1} &\left[ m_2 \circ \left( \sum_{i=2}^{n-2} (-1)^i m_i \circ  m_{n-i} \right) \right] \circ_{fb} h_B + \left[ \left( \sum_{i=2}^{n-2} (-1)^i m_i \circ m_{n-i} \right) \circ m_2 \right] \circ_{fb} h_B = \\
        &(-1)^{n-1} \left[ - \sum_{i=2}^{n-2} (-1)^i \sum_{j=2}^{i-1} (-1)^j \left( m_j \circ m_{i-j+1} \right) \circ m_{n-i+1} \right] \label{nomeprovvisorio4} \\
        &- (-1)^{n-1} \left[\sum_{i=3}^{n-1} m_i \circ \sum_{j=2}^{n-i} (-1)^j m_j \circ m_{n-i-j+2} \right] \circ_{fb} h_B \, ,
    \end{align}
    where we used Equation \eqref{associativitycomposition}. 
    The free branch composition simplifies the Equation above as follows: the only non-vanishing contributions to the first term are given by $j=2,i-1,$ whereas the non-vanishing contributions to the second term are $j=2,n-i$, thus leading to Equation \eqref{nomeprovvisorio4} being equal to
    \begin{align}
        & \left[ (-1)^{n}  \sum_{i=3}^{n-2} (-1)^i \left( m_2 \circ m_{i-1} + (-1)^{i-1} m_{i-1} \circ m_2 \right) \circ m_{n-i+1} \right] \circ_{fb} h_B \\
        & (-1)^{n}  \left[ \sum_{i=3}^{n-2} m_i \circ \left(  m_2 \circ m_{n-i} + (-1)^{n-i} m_{n-i} \circ m_2 \right) \right] \circ_{fb} h_B \, .
    \end{align}
    Moreover, we may have a further simplification due to the free branch composition. In particular, in order to have a free input, the second composition of the first term can be done on legs with $Z$ only and the first composition on the second term can be done on legs with $h_B$ only\footnote{We are just listing the terms which result in having a free input, i.e. relevant to the free branch composition.}. This leads to
    \begin{align}
        & \left[ (-1)^{n} \sum_{i=3}^{n-2} (-1)^i \left( m_2 \circ m_{i-1} + (-1)^{i-1} m_{i-1} \circ m_2 \right) \circ_Z m_{n-i+1} \right] \circ_{fb} h_B \\
        &+ (-1)^{n} \left[ \sum_{i=3}^{n-2} m_i \circ_{h_B} \left(  m_2 \circ  m_{n-i} + (-1)^{n-i} m_{n-i} \circ  m_2 \right) \right] \circ_{fb} h_B  \label{eqn:result3} \\
        & = \sum_{i=3}^{n-2} (-1)^{i} m_i \circ m_{n-i+1} \ ,
    \end{align}
    where we have used Equation \eqref{definitionproductmn} and $\circ_Z + \circ_{h_B} = \circ$ when composing any two (higher) products constructed with the prescription given above, see Equation \eqref{eqn:generalmn}. Gathering Equation \eqref{eqn:result1}, \eqref{eqn:result2} and \eqref{eqn:result3}, we have
    \begin{align}
        \partial_B m_n &= m_2 \circ m_{n-1} + (-1)^{n-1} m_{n-1} \circ m_2 + \sum_{i=3}^{n-2} (-1)^{i} m_i \circ m_{n-i+1} \\
        &= \sum_{i=2}^{n-1} (-1)^i m_i \circ m_{n-i} = \Ass_n \, ,
    \end{align}
    that completes the proof.
\end{proof}

\section{Homotopy Transfer, Homotopy Module-Induction and Deformations} \label{sec:HochDef}

In this section, we explicitly show that the $A_\infty$-algebras induced on $B$ via Homotopy Module-Induction are not ($A_\infty$) Hochschild deformations of the $A_\infty$-algebra induced via usual homotopy transfer. For a detailed description of ($A_\infty$) Hochschild cohomology, we refer to \cite{Penkava:1994mu}, where the cyclic cohomology is discussed as well\footnote{The introduction of an invariant, non-degenerate inner product is useful in many contexts, e.g., in Physics, in order to write down action functionals.}. 

\subsection{Homotopy Transfer and Homotopy Module-Induction}~

The main motivation to investigate infinitesimal Hochschild deformations of the Homotopy Module-Induced $A_\infty$-Algebra is the following result.
\begin{prop}\label{prop:HMISCquasiisoA}
    Let Theorem \ref{thm:Ainftyside} hold. Then the $A_\infty$-algebra $(B, m_n)$ is strictly quasi-isomorphic to the dga algebra $(A, \de_A, \wedge)$ seen as an $A_\infty$-algebra if and only if $k_1 \neq - k_2$. The quasi-isomorphism reads $f_1 \coloneqq (k_1 + k_2 ) Z$ and $f_i \coloneqq 0,$ for all $i \neq 1$.
\end{prop}
\begin{proof}
    We construct the strict quasi-isomorphism by choosing $f_1 \colon B \to A$ by $f_1 \coloneqq (k_1 + k_2) \, Z,$ which is a quasi-isomorphism by definition, and $f_i \coloneqq 0,$ for all $i \neq 1.$

     Let us start by showing that the $A_\infty$-morphism relation (see Equation \eqref{eqn:Ainftymorph}) holds for $n=2.$ This equation becomes
     \begin{align}
         (k_1 + k_2) Z \circ m_2 =(k_1 + k_2)^2 \wedge \circ (Z \otimes Z) \, 
     \end{align}
  since $\partial_{BA} Z = 0.$ This relation is 
   satisfied because of the Property \ref{it:ImZhomo} of Definition \ref{def:homotopydatamodule}.
    
    All the  higher $A_\infty$-morphism relations (see again Equation \eqref{eqn:Ainftymorph}) are trivially satisfied because $f_i =0,$ for all $i =0,$ and the $A_\infty$-algebra $A$ only admits a non-trivial 2-product, i.e. all the $n$-products with $n > 2$ vanish. 
\end{proof}

\begin{corollary}\label{coroll:quasi-isoSC}
The $A_\infty$-algebras $(B, m_n),$ obtained in Theorem \ref{thm:Ainftyside}, with $k_1 \neq - k_2$, and $(B, m^\hot_n)$ are quasi-isomorphic.
\end{corollary}
\begin{proof}
    The proof simply follows from the composition of the quasi-isomorphism defined in Proposition \ref{prop:HMISCquasiisoA} and the quasi-isomorphism between $(A,\de_A,\wedge)$ and $(B,m_{n}^{\hot})$.
\end{proof}

This says that if the side conditions are verified (and $k_1 \neq - k_2$), the $A_\infty$-algebras $(B,m_n)$ and $(B,m_n^\hot)$ are equivalent in the $A_\infty$-category. With an eye towards Example \ref{eg:supermani}, this means that on split supermanifolds the Homotopy-Module-Induction and the Homotopy-Transfer procedures will give rise to equivalent structures.

Let us now consider Theorem \ref{thm:weaksideorcoefficients}. In the case where the weak side conditions are met, it is not yet clear whether there exists a (non-strict) quasi-isomorphism between the $A_\infty$-algebras $(B,m_n)$ and $(B,m_n^\hot)$. This will be the subject of further investigation. On the other hand, one can see that if condition \ref{it:mainthm2} of Theorem \ref{thm:weaksideorcoefficients} is met, then there is an obstruction to the equivalence of these $A_\infty$-algebras:
\begin{prop}\label{prop:noquasi-iso}
    Let Theorem \ref{thm:weaksideorcoefficients} hold in the case of hypothesis \ref{it:mainthm2} and assume that the product $\wedge$ on $A$ induces a non-trivial product on $\sfH (A)$, i.e., $0 \neq [\wedge] \in \sfH \left( \Hom \left( A^{\otimes 2} , A \right) \right)$. Then the cohomology class of $\wedge$ represents an obstruction to the existence of a quasi-isomorphism $f_\bullet:(B,m_n)^{\otimes \bullet} \to (B,m_n^\hot)$.
\end{prop}
\begin{proof}
    Suppose an $A_\infty$-quasi-isomorphism $f_\bullet \colon (B,m_n)^{\otimes \bullet} \to (B,m_n^\hot)$ exists, then the maps $f_\bullet$ satisfy Equation \eqref{eqn:Ainftymorph}. Recall that, by Lemma \ref{lemma:HTWMSG}, $m_2 = \partial_B \tilde m_2$. By plugging this in Equation \eqref{eqn:Ainftymorph} for $n=2$, we get
    \begin{align}
        f_1 \circ m_2 = f_1 \circ \partial_B \tilde m_2 = m_2^\hot \circ f_1^{\otimes 2} + \partial_B f_2 \ \implies \ m_2^\hot \circ f_1^{\otimes 2} = \partial_B \lambda \ ,
    \end{align}
    where $\lambda \in \Hom \left( B^{\otimes 2} , B \right)$ of degree $-1$. Being $f_1$ a quasi-isomorphism, this implies that
    \begin{align}
        m_2^\hot = \partial_B \sigma \ ,
    \end{align}
      for some  $\sigma \in \Hom \left( B^{\otimes 2} , B \right)$, 
    which is true if and only if
    \begin{align}
        \wedge = \partial_A \rho \ , 
\end{align}
for some $\rho \in \Hom \left( A^{\otimes 2} , A \right)$. Hence $f_1$ can be a quasi-isomorphism if and only if $[\wedge]=0$.
\end{proof}

\subsection{Massey Products}~

In the previous paragraphs, we have discussed under which assumptions the $A_\infty$-algebras $(B,m_n)$ and $(B,m_n^\hot)$ are quasi-isomorphic, as shown in Corollary \ref{coroll:quasi-isoSC} and Proposition \ref{prop:noquasi-iso}. This naturally leads to considerations regarding the \emph{Massey $A_\infty$-algebras} that are induced on $\sfH (B)$ (see, e.g., \cite{vallette:hal-00757010}). These algebras are intrinsically related to intersection theory on manifolds and their \emph{formality} (see also, e.g., \cite{Deligne:1975}).

These structures are induced by homotopy transfer of $(B , m_n )$ and $(B, m_n^\hot)$, respectively, to $\sfH (B)$. In particular, consider the following setup of quasi-isomorphic complexes:
\begin{equation}\label{MasseyA}
	\begin{tikzcd}[every arrow/.append style={shift left}]
		\hspace{-.2cm} \left( A , \de_A \right) \arrow{r}{Y}  & \left( B , \de_B \right) \arrow[l, "Z"] \arrow{r}{p}  & \left( \sfH (B) , 0 \right) \arrow[l, "i"] \, ,
	\end{tikzcd}
\end{equation}
where the complexes $A$ and $B$ satisfy the assumptions of Definition \eqref{def:homotopydatamodule}. The \virgolette pivotal'' complex $B$ comes with two homotopies $h_B$ and $h'_B$ related to the two maps $i \circ p$ and $Y \circ Z$ as
\begin{align}
    \partial_B h_B = \unit_B - Y \circ Z \ , \ \partial_B h'_B = \unit_B - i \circ p \ .
\end{align}
In particular, the homotopy operator $h_B$ is the one used in the Homotopy Module-Induction to induce an $A_\infty$ algebra on $B$, while the homotopy operator $h'_B$ is used in the homotopy transfer construction from $B$ to $\sfH(B)$ to induce the Massey products. We will denote the $A_\infty$-algebras transferred from $(B,m_n)$ and $(B,m_n^\hot)$ to $\sfH (B)$ by $(\sfH (B), \mathsf{M}_n)$ and $\sfH (B), \mathsf{M}_n^\hot)$, respectively.

Let us consider the situation where Theorem \ref{thm:Ainftyside} holds: as a Corollary of Proposition \ref{prop:HMISCquasiisoA} we have
\begin{corollary}
    The $A_\infty$-algebras $(\sfH (B), \mathsf{M}_n)$ and $(\sfH (B), \mathsf{M}_n^\hot)$ are isomorphic.
\end{corollary}

On the other hand, consider the situation where Proposition \ref{prop:noquasi-iso} holds, we have
\begin{corollary}
    The $A_\infty$-algebras $(\sfH (B), \mathsf{M}_n)$ and $(\sfH (B), \mathsf{M}_n^\hot)$ are not isomorphic. In particular, $\mathsf{M}_2 = 0$.
\end{corollary}

An open question is left regarding the situation where Theorem \ref{thm:weaksideorcoefficients} holds in the case of weak side conditions. Determining whether $(\sfH (B), \mathsf{M}_n)$ and $(\sfH (B), \mathsf{M}_n^\hot)$ are isomorphic appears to be non-trivial. We defer this task to the later work \cite{Cremonini-future}. 

We shall now investigate whether the Homotopy Module-Induction can be seen as a deformation of the Homotopy transfer $A_\infty$-algebra.
To this, let us recall some basic facts about Hochschild cohomology.

\subsection{Infinitesimal Hochschild Deformations and Homotopy Module-Induction}~

Consider the complex $B$, equipped with the $A_\infty$-structure induced via homotopy transfer $\left( B , m_n^\hot \right)$ and define the Hochschild cochains $C^\bullet \left( B \right) \coloneqq \Hom \left( \mathsf{T} \left( B \right) , B \right),$ where $\mathsf{T}(B)$ is the tensor space of $B.$ We can use the algebraic structure on $B$ to define a coboundary operator, which is the extension of the operator $\partial_B \colon \Hom \left( \mathsf{T} \left( B \right) , B \right) \to \Hom \left( \mathsf{T} \left( B \right) , B \right)$, as \footnote{The signs are chosen according to the convention of Equations \eqref{HMTWSDB} and \eqref{HMTWSDBA}.}
\begin{align}\label{HDNIMBYA}
    \de^\sfH \mu \coloneqq \sum_{n=1}^\infty \left( (-1)^{n-1} m_n^\hot \circ \mu - (-1)^{\mu} \mu \circ m_n ^\hot \right) \ ,
\end{align}
for a given $\mu \in \Hom \left( \mathsf{T} \left( B \right) , B \right)$, called the \emph{Hochschild differential}, where $m_1^\hot= \de_B$ and the bracket is the extension to $\Hom(\mathsf{T}(B), B)$ of the commutator on $\Hom(B,B).$  
The nilpotency of $\de^\sfH$ follows from the $A_\infty$-relations \eqref{condo}. Given an element $\mu \in \Hom \left( \mathsf{T} \left( B \right) , B \right)$, we say that it is an \emph{infinitesimal deformation} of the $A_\infty$-algebra if $\de^\sfH \mu = 0$. If $\mu = \de^\sfH \nu$, for some $\nu \in \Hom \left( \mathsf{T} \left( B \right) , B \right)$, we say that the deformation is \emph{trivial}. It is then clear that non-trivial (infinitesimal) deformations are given by the cohomology groups determined by $\de^\sfH$; this cohomology is known as \emph{Hochschild cohomology} and it is denoted as $\sfH \sfH \left( B \right)$. 

In the previous sections, we have shown how to induce on the complex $B$ algebraic structures with a procedure leveraging on the $\Im(Z)$-bimodule structure; in this section, we address whether the Homotopy Module-Induced $A_\infty$-structures are related to the $A_\infty$-algebra obtained from the homotopy transfer procedure. In particular, we will discuss if the structures obtained in Section \ref{sect:HomotopyTransMod} are a Hochschild deformation (possibly trivial) of the usual one. It will turn out that this is not the case. 

\begin{prop} \label{prop:notHocdef}
Assume that Theorem \eqref{thm:Ainftyside} holds and let $(B, m_n^\hot)$ and $(B, m_n)$ be the $A_\infty$-algebras induced by the homotopy transfer and the Homotopy Module-Induction, respectively. Then $(B, m_n)$ is not an infinitesimal deformation of $(B, m_n^\hot).$
\end{prop}

\begin{proof}
Consider the map
\begin{align}\label{HDNIMBYB}
    \mu \coloneqq \sum_{i=2}^\infty \left( m_i - m_i^\hot \right) \in \Hom \left( \mathsf{T} \left( B \right) , B \right) \ ;
\end{align}
we need to verify whether $\de^\sfH \mu = 0$ holds. We will proceed by checking the equality in each degree. The first relation to check reads
\begin{align}\label{HDNIMBYC}
    \partial_B \left( m_2 - m_2^\hot \right) = 0 \ ,
\end{align}
which is true since both $m_2$ and $m_2^\hot$ are Leibniz-compatible with $\de_B$. The second relation reads
\begin{align}\label{HDNIMBYD}
    \partial_B (m_3-m_3^\hot) - m_2^\hot \circ ( m_2- m_2^\hot) - (m_2 - m_2^\hot) \circ m_2^\hot = 0 \,, 
\end{align}
which can be written as
\begin{align}\label{HDNIMBYE}
    \partial_B m_3 - m_2^\hot \circ m_2 - m_2 \circ m_2^\hot = 0 \,, 
\end{align}
where we used $\partial_B m_3^\hot =0$ and $m_2^\hot \circ m_2^\hot = 0,$ see Proposition \ref{prop:closedproducts}. 
By using the properties of the map $Z$, we have
\begin{align}\label{HDNIMBYF}
    m_2^\hot \circ m_2 \left( b \otimes b' \otimes b'' \right) = m_2^\hot \left( m_2 \left( b \otimes b' \right) \otimes b'' \right) - m_2^\hot \left( b \otimes m_2 \left( b' \otimes b'' \right) \right) = 0 \ .
\end{align}
On the other hand, we can explicitly verify that 
\begin{align}\label{HDNIMBYG}
    m_2 \circ m_2^\hot \left( b \otimes b' \otimes b'' \right) = & k_1 \left( Z ( b ) \wedge Z (b') \right) \rhd b'' - k_2 b \lhd \left( Z ( b' ) \wedge Z (b'') \right) \\
    & + (k_2 - k_1) Y \left( Z (b) \wedge Z (b') \wedge Z (b'') \right) \ .
\end{align}
Considering coefficients $k_1, \, k_2$ such that $k_1 \neq -1$ and $k_2 \neq -1$, Equation \eqref{HDNIMBYG} is different from $\partial_B m_3$, see Equation \eqref{HTWMSNA}, thus showing that $\mu$ does not correspond to a deformation. 

In the case $k_1=k_2= -1,$ we immediately see from Equations \eqref{HDNIMBYG} and \eqref{HTWMSNA} that Equation \eqref{HDNIMBYD} is verified. 
Thus, the next relation to check reads
\begin{align}\label{HDNIMBYH}
     \partial_B ( m_4 - m_4^\hot ) & - m_2^\hot \circ ( m_3 - m_3^\hot ) + (m_3 - m_3^\hot) \circ m_2^\hot \\
    & + m_3^\hot \circ ( m_2 - m_2^\hot ) - (m_2 - m_2^\hot) \circ m_3^\hot = 0 \ .
\end{align}
Note that we can use the $A_\infty$-relation for the homotopy transfer with $n=4$ together with Proposition \ref{prop:closedproducts} in order to cancel the terms having $m_n^\hot$ only. Moreover, $m_2^\hot \circ m_3 = 0$ because of the properties \ref{it:ImZhomo} of $Z$, together with $Z \circ h_B = 0, $ and $m_3 \circ m_2^\hot = 0$, because of $h_B \circ Y = 0$. Thus Equation \eqref{HDNIMBYH} can be written as
\begin{align}\label{HDNIMBYI}
    \partial_B m_4 + m_3^\hot \circ m_2 - m_2 \circ m_3^\hot = 0 \, , 
\end{align}

In order for  Equation \eqref{HDNIMBYI} to hold, one should verify that the terms containing $h_B$ (i.e., $\partial_B m_4$) and those containing $h_A$ (the last two) vanish separately. This would imply that $\partial_B m_4 = 0$, which is not satisfied. We conclude that
\begin{align}
    \de^\sfH \mu \neq 0 \ ,
\end{align}
so that there is no deformation connecting the homotopy transfer $A_\infty$-algebra and the Homotopy Module-Induced one.
\end{proof}

\begin{example}[\textbf{Associative algebras}]
    The failure of the infinitesimal deformation prescription happens even at the level of associative algebras. In particular, let $\left( B , m_2^\hot \right)$ and $\left( B , \left. m_2 \right|_{k_1=1,k_2=0} \right)$ be associative algebras, where the product of the latter algebra is defined by fixing the coefficients of Equation \eqref{HTWMSG} to $k_1=1,k_2=0.$ We assume the same conditions for the maps $Z$ and $Y$ as in Definition \eqref{def:homotopydatamodule}. In order to check whether the product $\left. m_2 \right|_{k_1=1,k_2=0}$ can be obtained via a deformation of $m_2^\hot$, we need to check if 
    \begin{align}
    \mu \coloneqq \left. m_2 \right|_{k_1=1,k_2=0} - m_2^\hot
    \end{align}
    is closed under the action of the Hochschild differential given by $\de^\sfH = \left[ m_2^\hot , \, \cdot \, \right],$ as discussed in the proof of Proposition \ref{prop:notHocdef}. Note that if the closure condition of $\mu$ with respect to $\de^\sfH$ held, then there would exist a one-parameter family of maps $m (t) \in \Hom \left( B^{\otimes 2} , B \right) [t]$ at most linear in $t$ such that $m(0) = m_2^\hot$ and $m(1)= \left. m_2 \right|_{k_1=1,k_2=0}$. Explicitly, we have
    \begin{align}
        \de^\sfH \mu  & = - m_2^\hot \circ ( \left. m_2 \right|_{k_1=1,k_2=0} - m_2^\hot) - (\left. m_2 \right|_{k_1=1,k_2=0} - m_2^\hot ) \circ m_2^\hot \\
        & = - m_2^\hot \circ \left. m_2 \right|_{k_1=1,k_2=0} - \left. m_2 \right|_{k_1=1,k_2=0} \circ m_2^\hot \ ,
    \end{align}
    which gives
    \begin{align}
        \de^\sfH \mu \left( b \otimes b' \otimes b'' \right)  = & - Y \left( Z \left( Z (b) \rhd b' \right) \wedge Z ( b'' ) \right) + Y \left( Z (b) \wedge Z \left( Z \left( b' \right) \rhd b'' \right) \right) \\
        & - Z \circ Y \left( Z ( b ) \wedge Z ( b' ) \right) \rhd b'' + Y \left( Z ( b ) \wedge Z ( b' ) \wedge Z ( b'' ) \right) \\
         = & - Z \circ Y \left( Z ( b ) \wedge Z ( b' ) \right) \rhd b'' + Y \left( Z ( b ) \wedge Z ( b' ) \wedge Z ( b'' ) \right) \neq 0 \ .
    \end{align}
    This shows that $\mu$ is not a Hochschild cocycle, thus there is no infinitesimal deformation connecting the two algebraic structures.
\end{example}

\printbibliography

\end{document}